\documentclass{article}

\usepackage{microtype}
\usepackage{graphicx}
\usepackage{subcaption}
\usepackage{booktabs} 
\usepackage{adjustbox}
\usepackage{multirow}
\usepackage{placeins}
\usepackage{needspace}
\usepackage[table]{xcolor}

\usepackage{hyperref}

\let\appendixaddcontentsline\addcontentsline
\usepackage[preprint]{icml2027}

\usepackage{amsmath}
\usepackage{amssymb}
\usepackage{mathtools}
\usepackage{amsthm}

\usepackage[capitalize,noabbrev]{cleveref}

\theoremstyle{plain}
\newtheorem{theorem}{Theorem}[section]
\newtheorem{proposition}[theorem]{Proposition}
\newtheorem{lemma}[theorem]{Lemma}
\newtheorem{corollary}[theorem]{Corollary}
\theoremstyle{definition}

\theoremstyle{remark}

\usepackage[textsize=tiny]{todonotes}

\icmltitlerunning{TensorSpectra: Bridging the Spectral and Tensor Worlds for High-Dimensional PDEs}

\begin{document}
\raggedbottom

\twocolumn[
  \icmltitle{TensorSpectra: Bridging the Spectral and Tensor Worlds \\ for High-Dimensional PDEs}



  \icmlsetsymbol{cor}{*}

  \begin{icmlauthorlist}
    \icmlauthor{Bingcheng Hu}{a}
    \icmlauthor{Tianchi Yu}{b,c,d}
    \icmlauthor{Zhaoxiang Li}{a,cor}
    \icmlauthor{Ivan Oseledets}{b,c,e}
  \end{icmlauthorlist}

  \icmlaffiliation{a}{Shanghai Normal University}
  \icmlaffiliation{b}{AXXX}
  \icmlaffiliation{c}{Applied AI Institute}
  \icmlaffiliation{d}{Southern University of Science and Technology}
  \icmlaffiliation{e}{INM}

  \icmlcorrespondingauthor{Zhaoxiang Li}{zxli@shnu.edu.cn}

  \icmlkeywords{High-dimensional PDEs, Spectral Methods, Low-rank Tensor Approximation, Physics-informed Machine Learning.}

  \vskip 0.3in
]



\printAffiliationsAndNotice{}

\begin{abstract}
  High-dimensional partial differential equations (PDEs) remain challenging for spectral methods due to the exponential growth of multidimensional coefficient spaces. We propose \textbf{TensorSpectra}, a low-rank spectral framework that directly parameterizes spectral coefficient tensors in canonical polyadic (CP) form, reducing representation complexity from $\mathcal{O}(N^d)$ to $\mathcal{O}(dNr)$. TensorSpectra combines structure-preserving spectral calculus, de-aliased nonlinear interactions, operator-aware residual weighting, and progressive low-rank corrections to enable efficient PDEs solving in compressed coefficient space. We further establish a theoretical framework connecting spectral truncation, low-rank approximation, and optimization errors. Extensive experiments across five PDE families and 24 benchmark configurations demonstrate \textbf{S}tate-\textbf{O}f-\textbf{T}he-\textbf{A}rt accuracy among four representative baselines, with TensorSpectra consistently achieving the lowest relative $L^2$ errors. More broadly, our findings highlight the potential of low-rank spectral formulations for high-dimensional PDEs, motivating further investigation of coefficient-space approaches and enriching the methodological landscape of numerical PDEs solving.
\end{abstract}

\section{Introduction}

\subsection{Background and Motivation}

High-dimensional partial differential equations (PDEs) arise naturally in many scientific and engineering problems, including uncertainty quantification, kinetic equations, stochastic dynamics, molecular simulation, and high-dimensional control \cite{Bungartz_Griebel_2004,han2018solving}. In these settings, the solution often depends on a large number of spatial, parametric, or latent variables, making accurate numerical approximation substantially more challenging than in low-dimensional regimes \cite{cohen2015approximation,hornung2025space}. Developing scalable solvers for such problems is therefore important not only for extending classical numerical methods to higher dimensions, but also for enabling reliable simulation and prediction in systems whose intrinsic state spaces are inherently high-dimensional. However, the computational cost of conventional discretization-based methods typically grows rapidly with dimension, making the efficient representation and solution of high-dimensional PDEs a persistent challenge in scientific computing \cite{khoromskij2012tensors}.

The above difficulties suggest that the scalability of high-dimensional PDE solvers depends critically on how the solution is represented. Spectral coordinates provide a natural description for smooth solutions and expose differential operators through explicit algebraic structure \cite{trefethen2000spectral,boyd2001chebyshev}. However, the resulting multidimensional coefficient field is itself a high-order tensor, whose dense representation still grows exponentially with dimension \cite{bigoni2016spectral}. A large ambient tensor does not necessarily imply equally high intrinsic complexity, since spectral coefficients may exhibit strong cross-dimensional separability and admit a compact low-rank representation \cite{bachmayr2023low}. This observation motivates a representation-centric question: \textbf{Can high-dimensional PDEs be solved directly in spectral space while replacing the exponentially large coefficient tensor with a structured low-rank representation?}

\subsection{Related Work}
\paragraph{Physics-informed neural PDE solvers.}
Physics-informed neural PDE solvers have evolved well beyond the original PINN formulation. Strong-form methods such as PINN \cite{raissi2019physics} and DGM \cite{sirignano2018dgm} minimize pointwise PDE residuals, while DRM \cite{yu2018deep}, VPINN \cite{kharazmi2019variational}, and $hp$-VPINNs \cite{kharazmi2021hp} introduce energy- or weak-form objectives; cPINN \cite{jagtap2020conservative} and XPINNs \cite{jagtap2020extended} further exploit domain decomposition. Recent work has also improved neural representations through adaptive activations, trainable Gaussian features such as PIG \cite{kang2025pig}, and residual-adaptive architectures such as PirateNets \cite{wang2024piratenets}. Nevertheless, optimization remains challenging because differential operators can induce ill-conditioned objectives and imbalance among PDE, initial, and boundary losses \cite{wang2021understanding}, motivating gradient-based balancing \cite{krishnapriyan2021characterizing}, NTK-guided weighting \cite{wang2022and}, and self-adaptive weighting \cite{mcclenny2023self}. Despite these advances in formulation, architecture, and optimization, most methods still rely on generic coordinate-to-solution neural representations whose internal structure is learned implicitly, leaving the explicit and dimension-scalable organization of high-dimensional PDE solutions largely unresolved.

\paragraph{Spectral methods and neural spectral solvers.} 
Spectral methods approximate PDE solutions using global orthogonal bases, with Fourier bases naturally suited to periodic problems and Jacobi families, including Legendre and Chebyshev polynomials, providing flexible representations for non-periodic domains \cite{shen2011spectral,boyd2001chebyshev,trefethen2000spectral}. For sufficiently smooth solutions, they offer rapid convergence and structured algebraic representations of differential operators in coefficient space \cite{gottlieb1977numerical}. The spectral bias of neural networks, which favors low-frequency components over high-frequency and multiscale structures \cite{rahaman2019spectral,tancik2020fourier}, has further motivated frequency-aware neural PDE solvers. Representative developments include FNN \cite{ngom2021fourier}, FNO \cite{li2021fourier}, and PINO \cite{li2024physics}, followed by spectral operator frameworks such as OPNO \cite{liu2024render} and SCLON \cite{choi2024spectral}. More recently, SINN \cite{yu2025spectral} formulated PDE solving directly in spectral coefficient space through algebraic differentiation, with subsequent extensions introducing coefficient scaling and basis-dependent embeddings to improve high-dimensional performance \cite{yu2026spectral}. However, these approaches do not explicitly exploit cross-dimensional low-rank structure in spectral coefficient tensors, potentially limiting their representational scalability as the dimensionality increases.

\paragraph{Tensor decompositions and separable representations.}
Low-rank tensor methods, including canonical polyadic (CP), Tucker, and tensor-train decompositions, reduce high-dimensional complexity by exploiting separability across coordinates \cite{kolda2009tensor,oseledets2011tensor,hackbusch2009new,grasedyck2013literature,pruliere2010deterministic}. This principle has also been adopted in neural PDE solvers. SPINN \cite{cho2023separable} constructs separable solution representations using coordinate-wise subnetworks, enabling efficient residual evaluation on tensor-product collocation grids. However, its reliance on physical-space automatic differentiation can introduce additional computational and memory costs, particularly for higher-order differential operators. LoRFS \cite{kaminskyovercoming} instead employs low-rank sums of separable Fourier expansions, allowing closed-form, sampling-free evaluation of certain physics-informed objectives and their gradients. While sharing the principle of spectral separability, our framework extends beyond Fourier representations by integrating low-rank coefficient tensors with spectral operator calculus across different basis families, operator-aware conditioning, and progressive residual correction. These developments highlight the complementary strengths of tensor factorization and spectral methods for high-dimensional PDE solving.

\begin{figure*}[t]
    \centering
    \includegraphics[width=\textwidth]{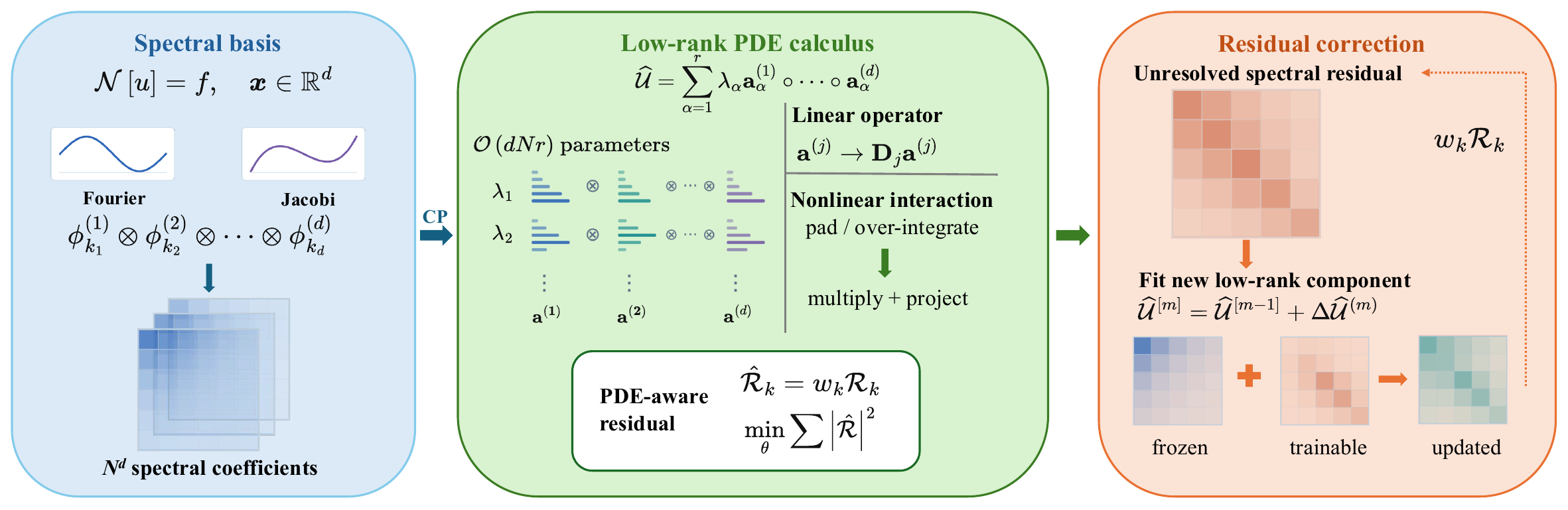}
    \caption{Tensorspectra framework.}
    \label{fig:teaser}
\end{figure*}

\subsection{Main Contributions}
Motivated by the complementary advantages of spectral computation and low-rank separable representations, we propose \textbf{TensorSpectra}, a unified coefficient-space framework for scalable high-dimensional PDE solving, as illustrated in Figure~\ref{fig:teaser}. Our central perspective is that a major computational bottleneck in high-dimensional spectral methods arises from the exponential growth of the full tensor-product coefficient representation, rather than from spectral discretization itself. Instead of explicitly constructing and subsequently decomposing the full coefficient tensor, TensorSpectra directly parameterizes spectral coefficients in a compact low-rank CP form and performs PDE computations through structured spectral operations on the individual factors. This formulation combines the algebraic advantages of spectral methods with the dimensional scalability of low-rank tensor representations. The main contributions are summarized as follows:
\begin{itemize}

\item \textbf{Low-rank spectral tensor representation.}
We represent the multidimensional spectral coefficient field directly in low-rank CP form, reducing the parameter complexity from $\mathcal{O}(N^d)$ to $\mathcal{O}(dNr)$ while explicitly capturing cross-dimensional spectral structure across different spectral bases.

\item \textbf{Structure-preserving spectral operator calculus.}
We apply differential operators directly to one-dimensional spectral factors, avoiding full tensor reconstruction and repeated spatial automatic differentiation, thereby reducing computational and memory overhead while improving training efficiency. This structured formulation further extends to nonlinear PDEs through de-aliased spectral interactions.



\item \textbf{Spectral--rank convergence analysis.}
We establish an error framework separating spectral truncation, low-rank approximation, and optimization errors, together with stability and convergence results that characterize the roles of spectral resolution, tensor rank, and progressive correction.
\end{itemize}

\section{Preliminaries}\label{sec:preliminaries}
In this section, we first introduce the general PDE formulation considered in this work and then establish the spectral representation and notation used throughout the paper, providing the necessary foundation for the TensorSpectra framework developed in the next section.

\subsection{Problem Setup}
We consider a general class of high-dimensional partial differential equations defined on a spatial domain $\Omega \subset \mathbb{R}^d$ over a time interval $[0,T]$. The solution $u(\boldsymbol{x},t)$ is governed by
\begin{equation} \label{eq:pde}
  \begin{aligned}
    \mathcal{N}[u](\boldsymbol{x},t) &= f(\boldsymbol{x},t), && (\boldsymbol{x},t)\in\Omega\times[0,T],\\
    u(\boldsymbol{x},t) &= g(\boldsymbol{x},t), && (\boldsymbol{x},t)\in\partial\Omega\times[0,T],\\
    u(\boldsymbol{x},0) &= h(\boldsymbol{x}), && \boldsymbol{x}\in\Omega,
  \end{aligned}
\end{equation}
where $\mathcal{N}$ denotes a general differential operator, $f$ is the forcing term, and $g$ and $h$ specify the boundary and initial conditions, respectively. This formulation accommodates both linear and nonlinear PDEs considered throughout this work.

For stationary problems, the temporal dependence and initial condition are omitted, and the governing equation reduces to $\mathcal{N}[u](\boldsymbol{x})=f(\boldsymbol{x})$. Although the formulation above is written with Dirichlet boundary data for simplicity, other boundary conditions can be incorporated through the corresponding boundary operator. Throughout this work, the main difficulty of interest is the high-dimensional dependence of $u$ on $\boldsymbol{x}$, which motivates seeking a representation that remains tractable as the dimension $d$ increases.

To obtain a representation that is compatible with high-dimensional computation, we express the solution in a tensor-product spectral basis rather than parameterizing it directly in physical space. This transforms the original PDE into an equivalent problem for the corresponding multidimensional spectral coefficients, where differentiation and other PDE operators can be treated through their algebraic action in coefficient space. We introduce this spectral representation and the associated notation in the next subsection.

\subsection{Spectral Representation}
For each coordinate direction $j=1,\ldots,d$, we introduce a one-dimensional spectral basis $\{\phi_{k_j}^{(j)}\}$ in the variable $x_j$. Their tensor product defines the multidimensional basis $\Phi_{\boldsymbol{k}}(\boldsymbol{x})=\prod_{j=1}^{d}\phi_{k_j}^{(j)}(x_j)$, where $\boldsymbol{k}=(k_1,\ldots,k_d)$ denotes the multidimensional spectral index. The solution is then approximated by the truncated expansion
\begin{equation}
  \begin{aligned}
    u(\boldsymbol{x},t) &\approx \sum_{\boldsymbol{k}\in\mathcal{K}} \widehat{u}_{\boldsymbol{k}}(t) \Phi_{\boldsymbol{k}}(\boldsymbol{x}) \\ &= \sum_{\boldsymbol{k}\in\mathcal{K}} \widehat{u}_{\boldsymbol{k}}(t) \prod_{j=1}^{d}\phi_{k_j}^{(j)}(x_j),
  \end{aligned}
\end{equation}
where $\mathcal{K}$ denotes the retained spectral index set and $\widehat{u}_{\boldsymbol{k}}(t)$ is the spectral coefficient associated with mode $\boldsymbol{k}$. The choice of the one-dimensional basis depends on the problem setting; for example, Fourier bases are naturally suited to periodic domains, while polynomial bases such as Jacobi polynomials can be used for non-periodic problems.

Collecting the coefficients associated with all retained multidimensional modes yields the spectral coefficient tensor
\begin{equation}
  \widehat{\mathcal U} = \left[ \widehat{u}_{k_1},\ldots,\widehat{u}_{k_d} \right] \in \mathbb{C}^{N_1\times\cdots\times N_d},
\end{equation}
where $N_j$, $j = 1,\ldots,d$ denotes the number of retained spectral modes along the $j$-th coordinate. When the same resolution $N$ is used in every dimension, this simplifies to $\widehat{\mathcal U}\in\mathbb{C}^{N\times\cdots\times N}$. Thus, the spectral representation of a high-dimensional PDE solution can be equivalently characterized by its multidimensional coefficient tensor $\widehat{\mathcal U}$.

Despite the approximation power of spectral methods, the multidimensional coefficient tensor becomes prohibitively large in high dimensions. If $N$ spectral modes are retained along each coordinate, the full tensor $\widehat{\mathcal U}$ contains $N^d$ entries, leading to exponential growth in both storage and computation as $d$ increases. This tensor-product bottleneck motivates a more compact representation of the spectral coefficient field itself. TensorSpectra addresses this issue by imposing an explicit low-rank tensor structure on $\widehat{\mathcal U}$, which is introduced in the next section.

\section{Methodology}\label{sec:ts-method}
In this section, we present TensorSpectra, a framework for solving high-dimensional PDEs by optimizing a structured spectral coefficient field. We introduce its CP representation with trainable temporal factors, factorwise spectral operator calculus, operator-aware residual objective, and additive low-rank refinement. Together, these components enable solution representation and residual evaluation through one-dimensional factors.

\subsection{Low-Rank Spectral Tensor Representation}
For time-dependent problems, we represent the spatial coefficient tensor for Equation~\eqref{eq:pde} as
\begin{equation}\label{eq:ts-temporal-cp}
  \begin{aligned}
  &\widehat{\mathcal U}(t) =\sum_{\alpha=1}^{r}\lambda_{\alpha}T_{\alpha}(t) \,\mathbf a_{\alpha}^{(1)}\circ\cdots\circ\mathbf a_{\alpha}^{(d)},
  \\ &\widehat{u}_{\boldsymbol{k}}(t) =\sum_{\alpha=1}^{r}\lambda_{\alpha}T_{\alpha}(t) \prod_{j=1}^{d}a_{\alpha,k_j}^{(j)}.
  \end{aligned}
\end{equation}
Here, $\circ$ denotes the vector outer product, $r$ is the prescribed CP rank, $\lambda_\alpha$ is a component amplitude, and $\mathbf a_\alpha^{(j)}\in\mathbb C^{N_j}$ contains one-dimensional spectral coefficients. The spatial factors are independent of time; the trainable factor $T_\alpha(t)$ carries the temporal dependence. Stationary problems are recovered by setting $T_\alpha\equiv1$. Each component therefore combines a separable spatial spectrum with its own temporal evolution. The trainable quantities are the amplitudes and one-dimensional factor coefficients; an individual mode is evaluated by the scalar product formula in Equation~\eqref{eq:ts-temporal-cp}, without assembling the complete tensor. Unlike separable neural representations that parameterize coordinate-wise functions through neural networks, TensorSpectra directly optimizes their spectral coefficients. This enables the separable structure to be preserved while differential operators are evaluated algebraically in coefficient space.

We expand each temporal factor in a finite basis,
\begin{equation}\label{eq:ts-time-basis}
  T_\alpha(t)=\sum_{q=0}^{p_t}c_{q,\alpha}\psi_q(t).
\end{equation}
Low-degree polynomials in normalized time, $\psi_q(t)=(t/T)^q$, provide a compact parameterization. Alternatively, collecting $\mathbf c_\alpha=(c_{0,\alpha},\ldots,c_{p_t,\alpha})^\top$ gives the space--time coefficient tensor
\begin{equation}\label{eq:ts-spacetime-cp}
  \widehat{\mathcal U}_{\mathrm{st}} =\sum_{\alpha=1}^{r}\lambda_\alpha \,\mathbf a_\alpha^{(1)}\circ\cdots\circ\mathbf a_\alpha^{(d)}\circ\mathbf c_\alpha.
\end{equation}
Equation~\eqref{eq:ts-temporal-cp} and Equation~\eqref{eq:ts-spacetime-cp} describe the same separated approximation that contracting the final coefficient axis with the temporal basis recovers the spatial tensor at any time. The first form supports residual evaluation at sampled times, while the second permits projection of the residual onto a temporal spectral basis. Both forms use $r(1+\sum_jN_j+p_t+1)$ parameters rather than a dense space--time tensor. Fourier factors are used for periodic coordinates and Jacobi-type factors for non-periodic coordinates; basis conventions are given in Appendix~\ref{sec:spectral basis}.

\subsection{Structure-Preserving Spectral Calculus}
Let $\mathbf D_j$ be the one-dimensional coefficient differentiation map in coordinate $x_j$. Spatial differentiation transforms only the corresponding CP factor, giving
\begin{equation}
  \partial_{x_j}\widehat{u}(t) =\sum_{\alpha=1}^{r}\lambda_\alpha T_\alpha(t) \,\mathbf a_\alpha^{(1)}\circ\cdots\circ (\mathbf D_j\mathbf a_\alpha^{(j)})\circ\cdots\circ\mathbf a_\alpha^{(d)}.
\end{equation}
Higher derivatives apply the corresponding higher-order differentiation map to one factor, while mixed derivatives transform each participating factor. Thus, spatial and temporal derivatives can be combined by acting independently on their respective axes. Each separable operator term preserves the rank bound; sums of terms can increase it. For Fourier bases, $\mathbf D_j=\operatorname{diag}(\mathrm{i}k_j)$, giving $\widehat{\Delta u}_{\boldsymbol{k}}=-\|\boldsymbol{k}\|_2^2\widehat{u}_{\boldsymbol{k}}$.

Time differentiation acts on the temporal factors according to
\begin{equation}\label{eq:ts-time-derivative}
  \partial_t^s \widehat{u}(t) =\sum_{\alpha=1}^{r}\lambda_\alpha T_\alpha^{(s)}(t) \,\mathbf a_\alpha^{(1)}\circ\cdots\circ\mathbf a_\alpha^{(d)},
  \quad s=1,2.
\end{equation}
The derivatives are obtained analytically from Equation~\eqref{eq:ts-time-basis}, including the physical-time scaling. In the space--time coefficient form, the equivalent operation replaces $\mathbf c_\alpha$ by $\mathbf D_t^s\mathbf c_\alpha$, where $\mathbf D_t$ is the temporal differentiation matrix. Differentiation with respect to physical time includes the scaling from the reference interval to $[0,T]$, so temporal resolution and the time horizon enter the operator consistently.

Polynomial nonlinearities are evaluated by expanding interactions between CP components and projecting their one-dimensional factor products. Fourier products use padding and truncation; Jacobi products use over-integration. Space--time products also include the temporal factors. A product of two rank-one components remains separable, allowing each coordinate product to be evaluated and projected independently before recombining the factors. The enlarged Fourier bandwidth or polynomial quadrature order prevents unresolved product modes from aliasing into the retained coefficients. These operations avoid dense tensor assembly, although nonlinear interactions can increase the rank. De-aliasing and interaction counts are detailed in Appendix~\ref{sec:de-aliased}.

\subsection{PDE-Aware Spectral Objective}
Let $\mathcal R_{\boldsymbol{k}}(t)=\widehat{\mathcal N[u]-f}_{\boldsymbol{k}}(t)$ be the spatial spectral residual.
The residual is formed by applying the spectral calculus to the current approximation and subtracting the projected forcing. Differential operators can amplify spectral components by different amounts, causing strongly amplified modes to dominate an unweighted objective even when their solution errors are small.

To compensate for operator-induced differences in residual scale, we apply a positive weight $w_{\boldsymbol{k}}$ based on the characteristic scale of the linear PDE operator. For a Fourier operator $\mathcal L$, this takes the form $w_{\boldsymbol{k}}=\rho(\sigma_{\mathcal L}(\boldsymbol{k}))$, with a decreasing scaling function $\rho$. For polynomial bases, differentiation mixes modes, so the weight is constructed from a coefficient-space operator scale rather than a diagonal Fourier symbol. The weight rescales the optimization objective while retaining the full PDE residual, including nonlinear terms. Given sampled mode--time pairs $\mathcal S$, the weighted loss is
\begin{equation}
  \mathcal L_{\mathrm{PDE}}^{\mathrm{prec}} =\frac{1}{|\mathcal S|}\sum_{(\boldsymbol{k},t)\in\mathcal S} |w_{\boldsymbol{k}}\mathcal R_{\boldsymbol{k}}(t)|^2.
\end{equation}
When time is projected onto a spectral basis, training instead uses the space--time residual coefficients in the objective
\begin{equation}
  \mathcal L_{\mathrm{PDE}}^{\mathrm{prec}} =\frac{1}{|\mathcal K||\mathcal Q|} \sum_{\boldsymbol{k}\in\mathcal K}\sum_{q\in\mathcal Q} |w_{\boldsymbol{k},q}\mathcal R_{\boldsymbol{k},q}|^2,
\end{equation}
where $\mathcal Q$ is the retained temporal index set. Here, $\mathcal R_{\boldsymbol{k},q}$ is the temporal projection of the spatial spectral residual. The weight may depend on both spatial indices and temporal degree. The sampled-time and projected-time objectives are two ways to train the temporal representation; their residual evaluations use the same factorwise operator construction.

For a PDE of temporal order $s_t$, prescribed initial coefficients $\widehat h_{\boldsymbol{k}}^{(\ell)}$ are enforced by
\begin{equation}
  \begin{aligned}
  &\mathcal L_{\mathrm{IC}} =\sum_{\ell=0}^{s_t-1}\frac{\gamma_\ell}{|\mathcal K_0|} \sum_{\boldsymbol{k}\in\mathcal K_0} \left|\partial_t^\ell\widehat u_{\boldsymbol{k}}(0) -\widehat h_{\boldsymbol{k}}^{(\ell)}\right|^2,
  \\ &\mathcal L=\mathcal L_{\mathrm{PDE}}^{\mathrm{prec}} +\lambda_{\mathrm{IC}}\mathcal L_{\mathrm{IC}}+\mathcal L_{\mathrm{con}}.
  \end{aligned}
\end{equation}
Here, $\gamma_\ell>0$ balance the initial data, and $\mathcal L_{\mathrm{con}}$ collects any boundary or problem-specific constraint penalties. Second-order evolution equations require both the initial value and its first time derivative. Terms enforced exactly by the basis or a hard parameterization are omitted; for example, $u=h_K+t\,v$ imposes a projected initial value $h_K$, with only the free component $v$ optimized. Stationary problems omit time and $\mathcal L_{\mathrm{IC}}$. PDE-specific residuals and weights appear in Appendix~\ref{sec:pde-loss}.

The representation, spectral calculus, and weighted objective above define the base TensorSpectra solver at a prescribed spectral resolution and CP rank. Its approximation error separates into spectral truncation, finite-rank approximation, and the gap between the trained coefficients and a near-best rank-$r$ approximation. The following bound makes these three contributions explicit in the orthonormal coefficient setting of Appendix~\ref{app:ts-proofs}.

\begin{theorem}[Approximation error]\label{thm:ts-decomp}
Let $u_N=P_Nu$ be the orthogonal spectral projection in $H$ and $\widehat{\mathcal U}_N$ its coefficient tensor. For a rank-$r$ approximation $\widehat{\mathcal V}_\eta$ within $\eta>0$ of the best rank-$r$ error $E_{\mathrm{rank}}(N,r)$, and a trained tensor $\widehat{\mathcal U}_\theta$ representing $u_\theta$,
\begin{equation}\label{eq:ts-decomp}
\begin{aligned}
\|u-u_\theta\|_H &\le \|u-u_N\|_H+E_{\mathrm{rank}}(N,r)+\eta\\
&\quad+\|\widehat{\mathcal V}_\eta-\widehat{\mathcal U}_\theta\|_F.
\end{aligned}
\end{equation}
\end{theorem}

\begin{table*}[t]
\centering
\caption{Comparison of relative $L^2$ errors among different physics-informed neural PDE solvers on high-dimensional linear and middle-dimensional nonlinear PDE benchmarks. All results are reported as mean ± standard deviation over three random seeds 100, 200 and 300. \textbf{Bold} values indicate the best performance.}
\label{tab:main_results}
\small
\begin{adjustbox}{max width=\linewidth}
\begin{tabular}{c|c|cccc|c}
\toprule
PDE & $d$ &
PINN &
SPINN &
PirateNet &
SINN &
\textbf{TensorSpectra}
\\
\midrule

\multirow{6}{*}{Poisson}
& 2
& 1.46e-03 ± 3.64e-04
& 3.61e-03 ± 6.33e-04
& 1.32e-03 ± 4.44e-05
& 3.03e-04 ± 9.20e-05
& \textbf{5.24e-08 ± 1.08e-09}
\\

& 5
& 2.91e-02 ± 3.36e-03
& 5.41e-03 ± 1.34e-03
& 3.10e-03 ± 2.30e-03
& 1.46e-04 ± 2.80e-05
& \textbf{6.98e-08 ± 8.59e-09}
\\

& 10
& 2.78e-02 ± 2.00e-03
& 1.02e-02 ± 4.06e-03
& 4.16e-02 ± 1.82e-03
& 1.14e-03 ± 8.84e-04
& \textbf{7.57e-08 ± 4.31e-09}
\\
& 30
& 9.99e-01 ± 9.74e-04
& $>1$
& $>1$
& 1.14e-04 ± 2.07e-05
& \textbf{9.39e-08 ± 2.52e-09}
\\
& 50
& 9.98e-01 ± 4.58e-04
& $>1$
& $>1$
& 2.97e-04 ± 7.00e-06
& \textbf{1.16e-07 ± 8.89e-09}
\\
& 100
& 9.98e-01 ± 1.66e-03
& $>1$
& $>1$
& 3.05e-04 ± 4.41e-05
& \textbf{9.06e-08 ± 1.52e-08}
\\

\midrule

\multirow{6}{*}{Heat}
& 2
& 6.49e-04 ± 1.09e-04
& 1.34e-03 ± 3.73e-05
& 2.54e-03 ± 1.20e-03
& 1.01e-04 ± 2.01e-05
& \textbf{3.84e-05 ± 1.80e-05}
\\
& 5
& 4.46e-03 ± 7.00e-04
& 3.15e-03 ± 1.29e-03
& 1.36e-03 ± 3.36e-05
& 1.79e-04 ± 8.05e-05
& \textbf{2.40e-05 ± 8.41e-06}
\\
& 10
& 8.56e-03 ± 1.53e-03
& 1.10e-02 ± 1.96e-03
& 9.04e-02 ± 1.25e-02
& 2.81e-04 ± 4.83e-05
& \textbf{2.26e-05 ± 1.63e-06}
\\
& 30
& 2.60e-02 ± 8.86e-03
& $>1$
& 9.99e-01 ± 1.32e-04
& 1.51e-03 ± 3.14e-04
& \textbf{6.56e-05 ± 1.44e-05}
\\
& 50
& 4.40e-02 ± 6.73e-04
& $>1$
& 9.98e-01 ± 3.06e-04
& 4.35e-03 ± 8.35e-04
& \textbf{1.04e-04 ± 2.24e-05}
\\
& 100
& 7.88e-02 ± 4.82e-03
& $>1$
& 9.98e-01 ± 6.73e-04
& 1.02e-02 ± 2.08e-03
& \textbf{3.82e-04 ± 6.40e-05}
\\

\midrule

\multirow{6}{*}{Wave}
& 2
& 9.03e-04 ± 1.74e-04
& 3.29e-03 ± 8.75e-04
& 1.68e-03 ± 3.14e-04
& 3.82e-02 ± 1.82e-03
& \textbf{6.04e-05 ± 1.17e-05}
\\
& 5
& 4.59e-03 ± 1.86e-04
& 3.60e-03 ± 1.66e-03
& 6.15e-03 ± 3.18e-03
& 6.91e-02 ± 3.23e-03
& \textbf{4.26e-05 ± 8.59e-06}
\\
& 10
& 1.49e-02 ± 5.68e-03
& 1.76e-02 ± 3.86e-03
& 4.51e-01 ± 3.21e-01
& 3.01e-01 ± 1.46e-01
& \textbf{5.45e-05 ± 1.16e-05}
\\
& 30
& 2.54e-02 ± 1.92e-03
& 9.86e-01 ± 1.76e-03
& 9.98e-01 ± 3.15e-04
& 9.75e-01 ± 1.41e-03
& \textbf{4.98e-04 ± 2.84e-04}
\\
& 50
& 4.65e-02 ± 1.35e-03
& $>1$
& 9.99e-01 ± 2.31e-04
& 9.79e-01 ± 1.00e-03
& \textbf{8.82e-04 ± 3.12e-04}
\\
& 100
& 8.81e-02 ± 4.98e-03
& $>1$
& 9.99e-01 ± 4.48e-04
& 9.66e-01 ± 6.11e-03
& \textbf{6.31e-03 ± 3.72e-03}
\\
\midrule

\multirow{3}{*}{AC}
& 2
& 2.57e-06 ± 1.97e-07
& 2.92e-06 ± 9.52e-07
& 1.12e-06 ± 2.38e-07
& 2.16e-06 ± 2.37e-07
& \textbf{2.46e-08 ± 1.18e-08}
\\
& 5
& 1.05e-05 ± 1.68e-06
& 7.39e-06 ± 3.40e-06
& 5.48e-06 ± 1.01e-06
& 3.71e-06 ± 6.87e-07
& \textbf{2.80e-07 ± 1.38e-07}
\\
& 10
& 1.71e-04 ± 1.09e-04
& 2.02e-05 ± 2.03e-05
& 4.16e-04 ± 1.97e-04
& 1.85e-05 ± 1.58e-05
& \textbf{3.57e-07 ± 7.02e-08}
\\
\midrule
\multirow{3}{*}{Reaction}
& 2
& 9.31e-04 ± 8.35e-05
& 1.18e-03 ± 1.04e-03
& 8.62e-01 ± 6.09e-01
& 1.17e-01 ± 1.29e-03
& \textbf{2.31e-04 ± 2.61e-05}
\\
& 5
& $>1$
& 1.01e-02 ± 2.06e-03
& $>1$
& 4.07e-01 ± 7.30e-02
& \textbf{4.23e-03 ± 4.64e-04}
\\
& 10
& $>1$
& 4.04e-01 ± 1.06e-01
& $>1$
& 4.32e-01 ± 7.56e-02
& \textbf{3.95e-03 ± 7.38e-04}
\\
\bottomrule
\end{tabular}
\end{adjustbox}
\end{table*}

\subsection{Low-Rank Spectral Residual Correction}
Theorem~\ref{thm:ts-decomp} identifies a limitation of training at a fixed rank: reducing the optimization gap does not remove the best rank-$r$ approximation error. This motivates increasing the available rank when the base model reaches an accuracy plateau. TensorSpectra does so through additive low-rank corrections, retaining the learned approximation while fitting new components to the remaining PDE residual. After training the base model, it freezes the learned parameters and adds a new low-rank component of the form
\begin{equation}
  \begin{aligned}
  &\Delta\widehat{\mathcal U}^{(m)}(t) =\sum_{\alpha=1}^{r_m}\delta\lambda_\alpha^{(m)} \delta T_\alpha^{(m)}(t) \,\delta\mathbf a_\alpha^{(m,1)}\circ\cdots\circ\delta\mathbf a_\alpha^{(m,d)},
  \\ &\widehat{\mathcal U}^{[m]}(t) =\widehat{\mathcal U}^{[m-1]}(t)+\Delta\widehat{\mathcal U}^{(m)}(t).
  \end{aligned}
\end{equation}
After $m$ stages, the accumulated rank is bounded by $r+\sum_{\ell=1}^{m}r_\ell$. Each increment has its own spatial and temporal factors, so refinement can improve both the spectrum and its evolution. Starting from the base model $\widehat{\mathcal U}^{[0]}$, only the new parameters $\boldsymbol\theta_m$ are optimized by minimizing the cumulative objective,
\begin{equation}
  \boldsymbol\theta_m^* =\arg\min_{\boldsymbol\theta_m} \mathcal L\!\left[\widehat{\mathcal U}^{[m-1]} +\Delta\widehat{\mathcal U}^{(m)}(\boldsymbol\theta_m)\right].
\end{equation}
The same full objective evaluates the cumulative solution, including any soft initial or boundary constraints. For linear PDEs, training seeks a correction whose operator image cancels the remaining residual; for nonlinear PDEs, the full nonlinear residual is recomputed without linearization. Hard-constrained models use increments satisfying the corresponding homogeneous constraints, such as $t\,\delta v$ for a fixed initial value. Temporal factors may also be stored as an additional CP axis; stationary corrections set them to one.

The correction stages enlarge the approximation space, but their practical objective remains the PDE residual of the cumulative solution. To interpret this objective, we need a complementary statement to Theorem~\ref{thm:ts-decomp}: under a stability inequality for a linear PDE, the full residual bounds the solution error. The following result applies to both the base approximation and any corrected iterate satisfying the prescribed essential constraints.

\begin{theorem}[Residual stability]\label{thm:ts-stability}
Let $Au=f$ and $\|v\|_X\le C_{\rm stab}\|Av\|_Y$ for all admissible errors. If $u_\theta$ satisfies the same essential constraints as $u$, then
\begin{equation}\label{eq:ts-full-res}
\|u-u_\theta\|_X\le C_{\rm stab}\|Au_\theta-f\|_Y.
\end{equation}
A sampled residual loss alone requires an independent bound on the unsampled residual.
\end{theorem}

Thus, residual reduction provides an accuracy certificate only when the stability and constraint assumptions hold and the full residual is controlled; a small sampled training loss alone is insufficient. Together, Theorems~\ref{thm:ts-decomp} and~\ref{thm:ts-stability} explain the roles of approximation capacity and residual minimization without asserting unconditional convergence of the correction stages. Proofs and full assumptions are given in Appendix~\ref{app:ts-proofs}.

\section{Numerical Experiments}\label{sec:experiments}
In this section, we evaluate TensorSpectra on linear and nonlinear PDEs, focusing on accuracy and dimensional scalability.

\paragraph{Benchmarks, baselines and implementation details.}
We consider the Poisson, heat, wave, steady-state Allen--Cahn, and nonlinear reaction equations, comparing against PINN \cite{raissi2019physics}, SPINN \cite{cho2023separable}, PirateNet \cite{wang2024piratenets}, and SINN \cite{yu2026spectral}. Runs use one RTX 4090 (24 GB) for $d<50$ or A100 (80 GB) for $d\geq50$. We implement TensorSpectra in JAX with CUDA 12, Equinox, Optax, and Adam \cite{kingma2014adam} is adopted as default optimizer. Benchmark definitions, baselines, and configurations are in Appendices~\ref{sec:linear bench}--\ref{sec:nonlinear bench}, \ref{sec:baseline methods}, and~\ref{sec:ex_config}.

\paragraph{Evaluation metrics.}
We use the relative $L^2$ error to evaluate solution accuracy. For time-independent and time-dependent PDEs respectively, we define
\begin{equation}
  \begin{aligned}
  &\text{Rel.} L^2  = \frac{\| u_{\text{pred}} - u_{\text{true}} \|_{L^2(\Omega)}}{\| u_{\text{true}} \|_{L^2(\Omega)}}, \\ &\text{Rel.} L^2 (T)  = \frac{\| u_{\text{pred}} (\cdot,T) - u_{\text{true}} (\cdot,T) \|_{L^2(\Omega)}}{\| u_{\text{true}} (\cdot,T) \|_{L^2(\Omega)}}.
  \end{aligned}
\end{equation}
For all time-dependent problems, we report the relative $L^2$ error at the terminal time $T$.

\paragraph{Main results.}
Table~\ref{tab:main_results} shows that TensorSpectra consistently achieves the lowest relative $L^2$ error across all reported benchmarks. On the linear PDEs, TensorSpectra maintains high accuracy up to
$d=100$, achieving errors of $9.06\times10^{-8}$, $3.82\times10^{-4}$ and $6.31\times10^{-3}$ on the Poisson, heat and wave equations respectively, while most baselines exhibit substantial accuracy degradation. The linear-benchmark errors vary non-monotonically with dimension, initially decreasing before rising at higher dimensions. This behavior may reflect the benefits of low-rank spectral factorization in mitigating computational complexity at moderate dimensions, while increasing approximation and optimization challenges become more pronounced in higher-dimensional settings. On the nonlinear Allen--Cahn and reaction benchmarks, TensorSpectra also consistently outperforms the competing methods, demonstrating its effectiveness beyond linear PDEs, and Figure~\ref{fig:loss} shows the training loss curves on nonlinear problems. Furthermore, Appendix~\ref{sec:comp cost} reports the computational cost per training update on the 30D wave benchmark, showing that TensorSpectra achieves high accuracy with substantially lower GFLOPs than most baselines.

\begin{table*}[t]
\centering
\caption{Comparison with LoRFS on the high-dimensional Reaction benchmark ($d=9$). \textbf{Bold} values indicate the best performance and \colorbox{green!20}{green} cells report the percentage improvement over the \underline{underlined} LoRFS value for each metric, using the best TensorSpectra basis. All the experiments are repeated over five random seeds and the mean ± standard deviation is reported.}
\label{tab:lorfs_reaction}
\small
\begin{adjustbox}{max width=\linewidth}
\begin{tabular}{lccc}
\toprule
Method & Basis & rMAE & rRMSE \\
\midrule
LoRFS
& Fourier 
& \underline{4.09e-02 ± 4.08e-03} 
& \underline{8.20e-03 ± 5.00e-04} \\

TensorSpectra 
& Legendre 
& \textbf{3.62e-03 ± 8.76e-04} 
& \textbf{2.36e-03 ± 1.15e-03} \\

TensorSpectra 
& Chebyshev
& 4.69e-03 ± 7.90e-04
& 2.93e-03 ± 1.40e-04\\

TensorSpectra 
& Jacobi $(\alpha=0.5, \beta=1.5)$
& 9.82e-03 ± 6.73e-03 
& 5.66e-03 ± 3.49e-03 \\

\midrule
Best improvement 
& -- 
& \cellcolor{green!20}91.15\% 
& \cellcolor{green!20}71.22\% \\
\bottomrule
\end{tabular}
\end{adjustbox}
\end{table*}

\begin{figure}[!htbp]
    \centering
    \includegraphics[width=\linewidth]{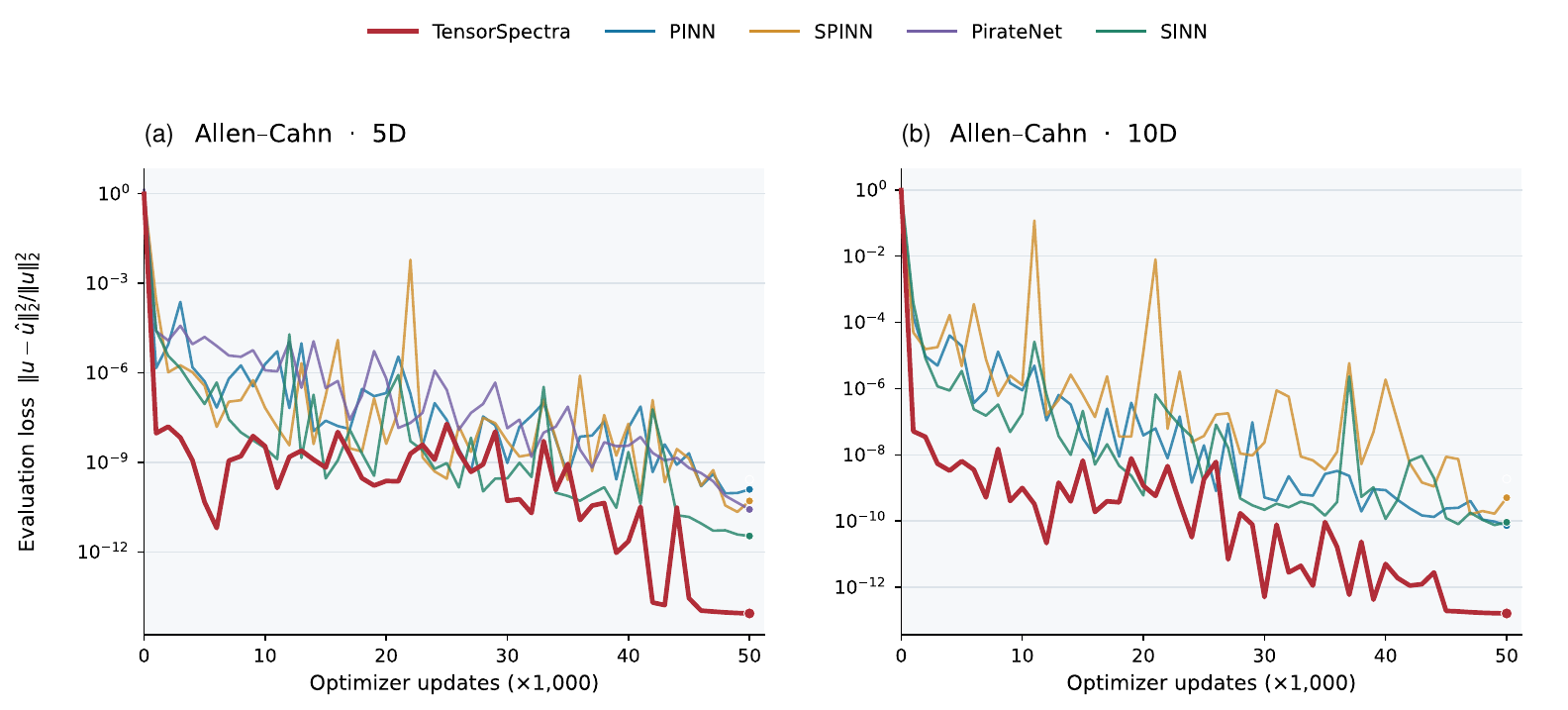}
    \caption{Optimization trajectories on nonlinear PDE benchmarks. Evaluation loss versus optimizer updates for (a) 5D Allen--Cahn and (b) 10D Allen--Cahn. TensorSpectra consistently converges to lower evaluation loss than the competing baselines.}
    \label{fig:loss}
\end{figure}

\section{Ablation Study and Hyperparameter Analysis}\label{sec:ablation}
In this section, we examine spectral basis compatibility, residual correction, and spectral weighting, then study sensitivity to temporal degree and CP rank.

\subsection{Ablation Study}
\paragraph{Compatibility with LoRFS.}
Table~\ref{tab:lorfs_reaction} compares TensorSpectra with published LoRFS results on the same Reaction benchmark (Appendix~\ref{sec:nonlinear bench}). Legendre, Chebyshev, and Jacobi variants all outperform Fourier-based LoRFS in rMAE and rRMSE. Legendre gives the best errors, $3.62\times10^{-3}$ and $2.36\times10^{-3}$, improving on LoRFS by $91.15\%$ and $71.22\%$. Supporting multiple orthogonal polynomial bases also enables bounded non-periodic problems without periodic extensions.

\paragraph{TensorSpectra component ablation.}
We add components to a rank-8 backbone under a fixed budget of 50k epochs: either base training alone or 25k base plus 25k correction epochs. Results average seeds 100, 200, and 300. In Figure~\ref{fig:ablation-progressive-reduction}, a rank-4 correction reduces relative $L^2$ error from $3.96\times10^{-3}$ to $3.68\times10^{-3}$; adding spectral weighting further reduces it to $2.68\times10^{-3}$, about $32\%$ below the backbone.

\begin{figure}[!htbp]
    \centering
    \includegraphics[width=1\linewidth]{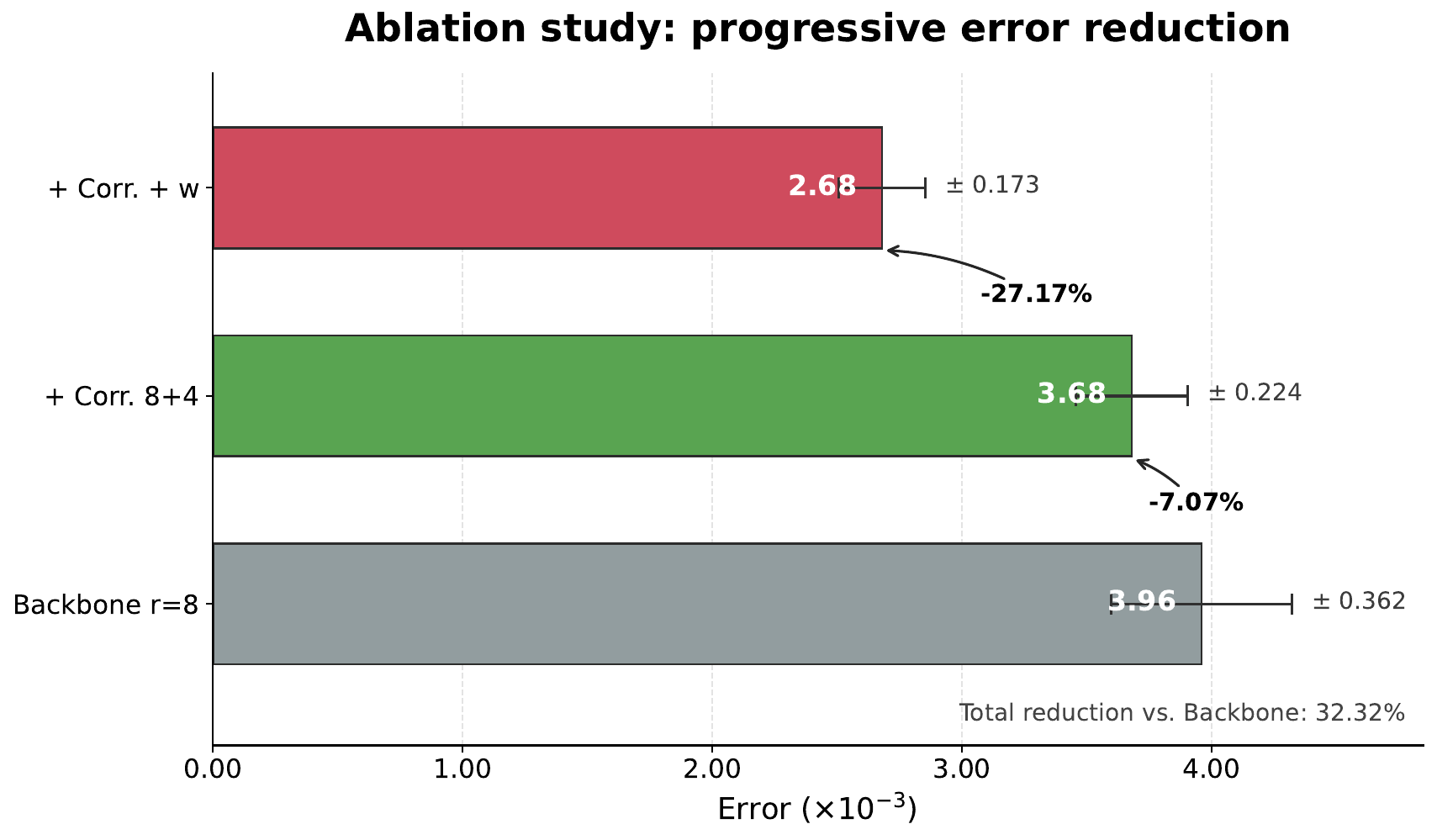}
    \caption{Ablation study of TensorSpectra on the contribution of residual correction and spectral weighting. Starting from a rank 8 backbone, we progressively add a rank 4 correction stage and the spectral weighting $w$. Bars report the mean relative $L^2$ error over three random seeds, and error bars indicate the corresponding standard deviation. All configurations use the same total training budget of 50k epochs.}
    \label{fig:ablation-progressive-reduction}
\end{figure}

\paragraph{Progressive rank correction and saturation analysis}
We investigate the effect of successive low-rank residual corrections on the 30-dimensional Wave benchmark. Starting from a rank-8 backbone, we progressively introduce correction components with ranks 4, 4, 2, 2, 2, 2, yielding cumulative ranks of 8, 12, 16, 18, 20, 22 and 24 respectively. Rather than training a higher-rank model from scratch, each correction stage freezes all previously learned components and optimizes only the newly introduced low-rank factors. As shown in Figure~\ref{fig:wave30d_progressive_rank}, the relative $L^2$ error initially decreases from $5.93\times10^{-4}$ to $3.74\times10^{-4}$ at the $8+4+4$ configuration, corresponding to a 36.9\% reduction relative to the rank-8 backbone. Further corrections, however, do not improve the mean accuracy beyond this point and instead exhibit non-monotonic fluctuations. These results suggest that progressive rank correction provides substantial gains in its early stages, while additional corrections yield diminishing returns under the considered training protocol.

\begin{figure}[!htbp]
    \centering
    \includegraphics[width=1\linewidth]{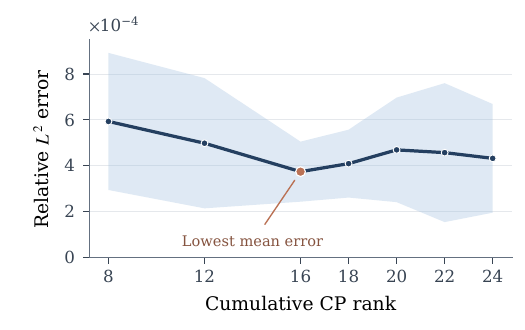}
    \caption{Progressive rank correction on the 30-dimensional Wave benchmark. Correction ranks \(4,4,2,2,2,2\) are sequentially added to a rank-8 backbone. The solid line shows the mean relative \(L^2\) error, with the shaded region indicating ± one standard deviation over three random seeds.}
    \label{fig:wave30d_progressive_rank}
\end{figure}

\subsection{Hyperparameter Analysis}

\paragraph{Sensitivity to time degree.}
We evaluate the temporal polynomial degree under a fixed TensorSpectra configuration with a rank 8 backbone and a rank 4 correction stage. For all settings, the backbone and correction stages are trained for 50k epochs each, resulting in the same total budget of 100k epochs. Figure~\ref{fig:ablation_time_degree_rank} (left) shows that the performance is not monotonic with respect to the temporal degree. Degrees 4 and 8 achieve comparable errors of approximately $5\times10^{-4}$, whereas increasing the degree to 16 or 24 leads to larger errors on the order of $10^{-3}$. A moderate temporal degree achieves the best accuracy among the tested settings.

\paragraph{Sensitivity to rank.}
We investigate the effect of the backbone and correction ranks. For all configurations, the backbone and correction stages are trained for 50k epochs each, resulting in the same total training budget of 100k epochs. Figure~\ref{fig:ablation_time_degree_rank} (right) shows that increasing the configuration from $4+2$ to $8+4$ reduces the relative $L^2$ error from $8.84\times10^{-4}$ to $5.84\times10^{-4}$, while further increasing the backbone rank to $12$ yields the best performance of $3.54\times10^{-4}$ with the $12+4$ configuration. These results indicate that increasing the representation capacity is beneficial up to a moderate rank, whereas excessive correction rank does not provide further improvement.

\begin{figure}[!htbp]
    \centering
    \includegraphics[width=1\linewidth]{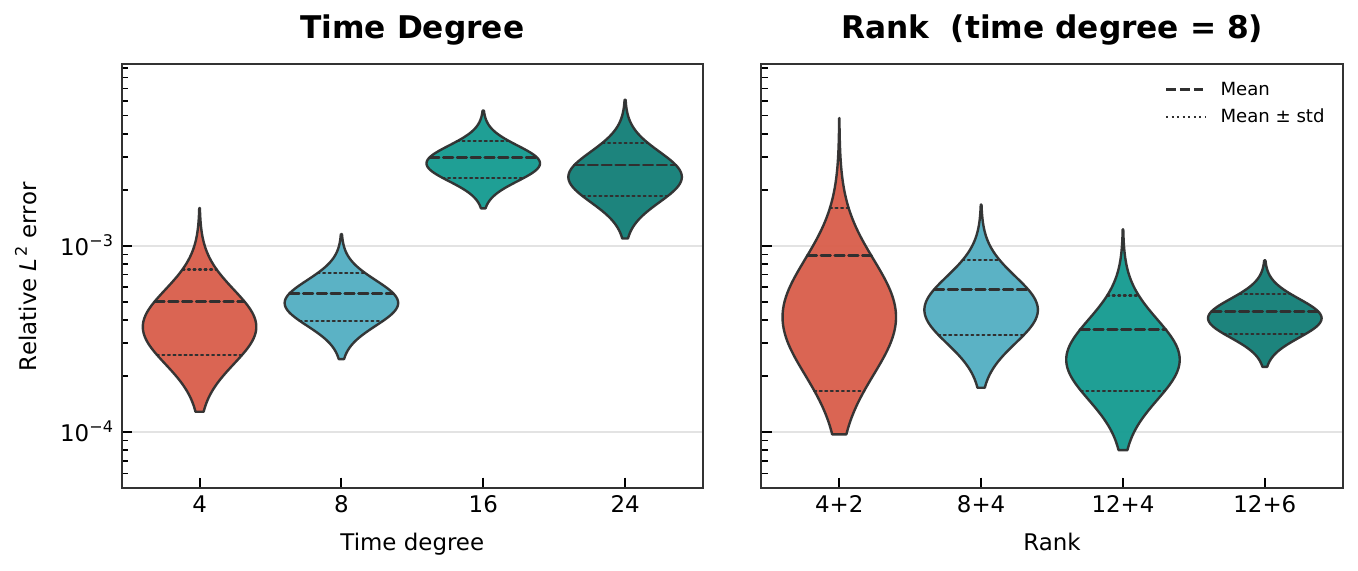}
    \caption{Relative $L^2$ error for different temporal degrees (left) and ranks (right). Dashed and dotted lines show the mean and mean $\pm$ standard deviation.}
    \label{fig:ablation_time_degree_rank}
\end{figure}

\FloatBarrier
\section{Discussions and Conclusions}\label{sec:conclusions}
Our results suggest that high-dimensional spectral PDE solving is fundamentally constrained by the representation complexity of the multidimensional coefficient field. By exploiting low-rank structure directly in spectral space, TensorSpectra retains the accuracy and operator structure of spectral methods while substantially reducing this complexity. More broadly, these results indicate that explicitly encoding solution structure can provide an effective alternative to relying entirely on generic neural representations to discover such structure through optimization.

The efficiency of TensorSpectra depends on spectral compressibility. Strongly non-separable dynamics, singular solutions, and nonlinear interactions can increase the required rank. Future work will therefore focus on adaptive rank control, more flexible tensor formats, and extensions to more challenging nonlinear and inverse problems. These directions may further clarify when low-rank spectral structure provides a practical route to scalable high-dimensional scientific computing.

\section*{AI Use Statement}
In this work, we used generative AI tools, including Iteris \cite{chen2026iteris}, an agentic research system for computational mathematics, to assist in developing and drafting proofs of mathematical claims, provide feedback on methodological exposition, and refine the interpretation and presentation of results. Additionally, these tools were used to polish the writing and improve clarity and organization. We carefully reviewed all AI-assisted work, including checking the assumptions, mathematical statements, and individual steps of the proofs, and verifying that the methodological descriptions and interpretations are consistent with our implementation and experimental evidence. We take full responsibility for the final content of this work, including all text, claims, and artifacts produced with the aid of generative AI.

\section*{Impact Statement}
This paper presents work whose goal is to advance the field of machine learning. There are many potential societal consequences of our work, none of which we feel must be specifically highlighted here.

\section*{Data Availability}
The datasets generated during and/or analysed during the current study are available from the corresponding author on request.

\section*{Conflict of interest}
The authors declared that they have no conflict of interest.
 
\section*{Acknowledgments}
Zhaoxiang Li is supported by the National Natural Science Foundation of China (12671434, 12571391, 12271366).

\bibliography{example_paper}

@article{raissi2019physics,
  title={Physics-informed neural networks: A deep learning framework for solving forward and inverse problems involving nonlinear partial differential equations},
  author={Raissi, Maziar and Perdikaris, Paris and Karniadakis, George E},
  journal={Journal of Computational physics},
  volume={378},
  pages={686--707},
  year={2019},
  publisher={Elsevier}
}

@article{cho2023separable,
  title={Separable physics-informed neural networks},
  author={Cho, Junwoo and Nam, Seungtae and Yang, Hyunmo and Yun, Seok-Bae and Hong, Youngjoon and Park, Eunbyung},
  journal={Advances in Neural Information Processing Systems},
  volume={36},
  pages={23761--23788},
  year={2023}
}

@article{wang2024piratenets,
  title={Piratenets: Physics-informed deep learning with residual adaptive networks},
  author={Wang, Sifan and Li, Bowen and Chen, Yuhan and Perdikaris, Paris},
  journal={Journal of Machine Learning Research},
  volume={25},
  number={402},
  pages={1--51},
  year={2024}
}

@inproceedings{yu2026spectral,
title={Spectral-Informed Neural Networks Outperform Spectral methods in High-dimensional {PDE}s},
author={Tianchi Yu and Ivan Oseledets},
booktitle={Forty-third International Conference on Machine Learning},
year={2026}
}

@inproceedings{kaminskyovercoming,
  title={Overcoming PINNs Failure Modes In High Dimension With Low-Rank Fourier Sum},
  author={Kaminsky, Natan and Freedman, Daniel and Radinsky, Kira},
  booktitle={Forty-third International Conference on Machine Learning},
  year={2026}
}

@article{Bungartz_Griebel_2004, title={Sparse grids}, volume={13}, DOI={10.1017/S0962492904000182}, journal={Acta Numerica}, author={Bungartz, Hans-Joachim and Griebel, Michael}, year={2004}, pages={147-269}}

@article{han2018solving,
  title={Solving high-dimensional partial differential equations using deep learning},
  author={Han, Jiequn and Jentzen, Arnulf and E, Weinan},
  journal={Proceedings of the National Academy of Sciences},
  volume={115},
  number={34},
  pages={8505--8510},
  year={2018},
  publisher={National Academy of Sciences}
}

@article{cohen2015approximation,
  title={Approximation of high-dimensional parametric PDEs},
  author={Cohen, Albert and DeVore, Ronald},
  journal={Acta Numerica},
  volume={24},
  pages={1--159},
  year={2015},
  publisher={Cambridge University Press}
}

@article{hornung2025space,
  title={Space-time deep neural network approximations for high-dimensional partial differential equations},
  author={Hornung, Fabian and Jentzen, Arnulf and Salimova, Diyora},
  journal={Journal of Computational Mathematics},
  volume={43},
  number={4},
  pages={918--975},
  year={2025}
}

@article{khoromskij2012tensors,
  title={Tensors-structured numerical methods in scientific computing: Survey on recent advances},
  author={Khoromskij, Boris N},
  journal={Chemometrics and intelligent laboratory systems},
  volume={110},
  number={1},
  pages={1--19},
  year={2012},
  publisher={Elsevier}
}

@book{trefethen2000spectral,
  title={Spectral methods in MATLAB},
  author={Trefethen, Lloyd N},
  year={2000},
  publisher={SIAM}
}

@book{boyd2001chebyshev,
  title={Chebyshev and Fourier spectral methods},
  author={Boyd, John P},
  year={2001},
  publisher={Courier Corporation}
}

@article{bigoni2016spectral,
  title={Spectral tensor-train decomposition},
  author={Bigoni, Daniele and Engsig-Karup, Allan P and Marzouk, Youssef M},
  journal={SIAM Journal on Scientific Computing},
  volume={38},
  number={4},
  pages={A2405--A2439},
  year={2016},
  publisher={SIAM}
}

@article{bachmayr2023low,
  title={Low-rank tensor methods for partial differential equations},
  author={Bachmayr, Markus},
  journal={Acta Numerica},
  volume={32},
  pages={1--121},
  year={2023},
  publisher={Cambridge University Press}
}

@article{sirignano2018dgm,
  title={DGM: A deep learning algorithm for solving partial differential equations},
  author={Sirignano, Justin and Spiliopoulos, Konstantinos},
  journal={Journal of computational physics},
  volume={375},
  pages={1339--1364},
  year={2018},
  publisher={Elsevier}
}

@article{yu2018deep,
  title={The deep Ritz method: a deep learning-based numerical algorithm for solving variational problems},
  author={Weinan E and Yu, Bing},
  journal={Communications in Mathematics and Statistics},
  volume={6},
  number={1},
  pages={1--12},
  year={2018},
  publisher={Springer}
}

@article{krishnapriyan2021characterizing,
  title={Characterizing possible failure modes in physics-informed neural networks},
  author={Krishnapriyan, Aditi and Gholami, Amir and Zhe, Shandian and Kirby, Robert and Mahoney, Michael},
  journal={Advances in Neural Information Processing Systems},
  volume={34},
  pages={26548--26560},
  year={2021}
}

@article{wang2022and,
  title={When and why PINNs fail to train: A neural tangent kernel perspective},
  author={Wang, Sifan and Yu, Xinling and Perdikaris, Paris},
  journal={Journal of Computational Physics},
  volume={449},
  pages={110768},
  year={2022},
  publisher={Elsevier}
}

@book{shen2011spectral,
  title={Spectral methods: algorithms, analysis and applications},
  author={Shen, Jie and Tang, Tao and Wang, Li-Lian},
  year={2011},
  publisher={Springer Science \& Business Media}
}

@article{ngom2021fourier,
  title={Fourier neural networks as function approximators and differential equation solvers},
  author={Ngom, Marieme and Marin, Oana},
  journal={Statistical Analysis and Data Mining: The ASA Data Science Journal},
  volume={14},
  number={6},
  pages={647--661},
  year={2021},
  publisher={Wiley Online Library}
}

@inproceedings{li2021fourier,
title={Fourier Neural Operator for Parametric Partial Differential Equations},
author={Zongyi Li and Nikola Borislavov Kovachki and Kamyar Azizzadenesheli and Burigede liu and Kaushik Bhattacharya and Andrew Stuart and Anima Anandkumar},
booktitle={International Conference on Learning Representations},
year={2021}
}

@article{yu2025spectral,
  title={Spectral informed neural networks},
  author={Yu, Tianchi and Qi, Yiming and Oseledets, Ivan and Chen, Shiyi},
  journal={Journal of Computational and Applied Mathematics},
  pages={117178},
  year={2025},
  publisher={Elsevier}
}

@article{kolda2009tensor,
  title={Tensor decompositions and applications},
  author={Kolda, Tamara G and Bader, Brett W},
  journal={SIAM review},
  volume={51},
  number={3},
  pages={455--500},
  year={2009},
  publisher={SIAM}
}

@article{oseledets2011tensor,
  title={Tensor-train decomposition},
  author={Oseledets, Ivan V},
  journal={SIAM Journal on Scientific Computing},
  volume={33},
  number={5},
  pages={2295--2317},
  year={2011},
  publisher={SIAM}
}

@article{hackbusch2009new,
  title={A new scheme for the tensor representation},
  author={Hackbusch, Wolfgang and K{\"u}hn, Stefan},
  journal={Journal of Fourier analysis and applications},
  volume={15},
  number={5},
  pages={706--722},
  year={2009},
  publisher={Springer}
}

@article{grasedyck2013literature,
  title={A literature survey of low-rank tensor approximation techniques},
  author={Grasedyck, Lars and Kressner, Daniel and Tobler, Christine},
  journal={GAMM-Mitteilungen},
  volume={36},
  number={1},
  pages={53--78},
  year={2013},
  publisher={Wiley Online Library}
}

@article{pruliere2010deterministic,
  title={On the deterministic solution of multidimensional parametric models using the proper generalized decomposition},
  author={Pruliere, Etienne and Chinesta, Francisco and Ammar, Amine},
  journal={Mathematics and Computers in Simulation},
  volume={81},
  number={4},
  pages={791--810},
  year={2010},
  publisher={Elsevier}
}

@article{kharazmi2019variational,
  title={Variational physics-informed neural networks for solving partial differential equations},
  author={Kharazmi, Ehsan and Zhang, Zhongqiang and Karniadakis, George Em},
  journal={arXiv preprint arXiv:1912.00873},
  year={2019}
}

@article{kharazmi2021hp,
  title={hp-VPINNs: Variational physics-informed neural networks with domain decomposition},
  author={Kharazmi, Ehsan and Zhang, Zhongqiang and Karniadakis, George Em},
  journal={Computer Methods in Applied Mechanics and Engineering},
  volume={374},
  pages={113547},
  year={2021},
  publisher={Elsevier}
}

@article{jagtap2020conservative,
  title={Conservative physics-informed neural networks on discrete domains for conservation laws: Applications to forward and inverse problems},
  author={Jagtap, Ameya D and Kharazmi, Ehsan and Karniadakis, George Em},
  journal={Computer Methods in Applied Mechanics and Engineering},
  volume={365},
  pages={113028},
  year={2020},
  publisher={Elsevier}
}

@article{jagtap2020extended,
  title={Extended physics-informed neural networks (XPINNs): A generalized space-time domain decomposition based deep learning framework for nonlinear partial differential equations},
  author={Jagtap, Ameya D and Karniadakis, George Em},
  journal={Communications in Computational Physics},
  volume={28},
  number={5},
  year={2020},
  publisher={Brown Univ., Providence, RI (United States)}
}

@inproceedings{kang2025pig,
  title={Pig: Physics-informed gaussians as adaptive parametric mesh representations},
  author={Kang, Namgyu and Oh, Jaemin and Hong, Youngjoon and Park, Eunbyung},
  booktitle={International Conference on Learning Representations},
  volume={2025},
  pages={73022--73046},
  year={2025}
}

@article{mcclenny2023self,
  title={Self-adaptive physics-informed neural networks},
  author={McClenny, Levi D and Braga-Neto, Ulisses M},
  journal={Journal of Computational Physics},
  volume={474},
  pages={111722},
  year={2023},
  publisher={Elsevier}
}

@article{wang2021understanding,
  title={Understanding and mitigating gradient flow pathologies in physics-informed neural networks},
  author={Wang, Sifan and Teng, Yujun and Perdikaris, Paris},
  journal={SIAM Journal on Scientific Computing},
  volume={43},
  number={5},
  pages={A3055--A3081},
  year={2021},
  publisher={SIAM}
}

@book{gottlieb1977numerical,
  title={Numerical analysis of spectral methods: theory and applications},
  author={Gottlieb, David and Orszag, Steven A},
  year={1977},
  publisher={SIAM}
}

@inproceedings{rahaman2019spectral,
  title={On the spectral bias of neural networks},
  author={Rahaman, Nasim and Baratin, Aristide and Arpit, Devansh and Draxler, Felix and Lin, Min and Hamprecht, Fred and Bengio, Yoshua and Courville, Aaron},
  booktitle={International conference on machine learning},
  pages={5301--5310},
  year={2019},
  organization={PMLR}
}

@article{tancik2020fourier,
  title={Fourier features let networks learn high frequency functions in low dimensional domains},
  author={Tancik, Matthew and Srinivasan, Pratul and Mildenhall, Ben and Fridovich-Keil, Sara and Raghavan, Nithin and Singhal, Utkarsh and Ramamoorthi, Ravi and Barron, Jonathan and Ng, Ren},
  journal={Advances in neural information processing systems},
  volume={33},
  pages={7537--7547},
  year={2020}
}

@article{li2024physics,
  title={Physics-informed neural operator for learning partial differential equations},
  author={Li, Zongyi and Zheng, Hongkai and Kovachki, Nikola and Jin, David and Chen, Haoxuan and Liu, Burigede and Azizzadenesheli, Kamyar and Anandkumar, Anima},
  journal={ACM/IMS Journal of Data Science},
  volume={1},
  number={3},
  pages={1--27},
  year={2024},
  publisher={ACM New York, NY}
}

@article{liu2024render,
  title={Render unto numerics: Orthogonal polynomial neural operator for PDEs with nonperiodic boundary conditions},
  author={Liu, Ziyuan and Wang, Haifeng and Zhang, Hong and Bao, Kaijun and Qian, Xu and Song, Songhe},
  journal={SIAM Journal on Scientific Computing},
  volume={46},
  number={4},
  pages={C323--C348},
  year={2024},
  publisher={SIAM}
}

@article{choi2024spectral,
  title={Spectral operator learning for parametric PDEs without data reliance},
  author={Choi, Junho and Yun, Taehyun and Kim, Namjung and Hong, Youngjoon},
  journal={Computer Methods in Applied Mechanics and Engineering},
  volume={420},
  pages={116678},
  year={2024},
  publisher={Elsevier}
}

@article{kingma2014adam,
  title={Adam: A method for stochastic optimization},
  author={Kingma, Diederik P and Ba, Jimmy},
  journal={arXiv preprint arXiv:1412.6980},
  year={2014}
}

@article{chen2026iteris,
  title={Iteris: Agentic research loops for computational mathematics},
  author={Chen, Leheng and Liu, Zihao and He, Wanyi and Dong, Bin},
  journal={arXiv preprint arXiv:2606.02484},
  year={2026}
}
\bibliographystyle{icml2027}

\newpage
\appendix
\onecolumn
\let\addcontentsline\appendixaddcontentsline

\addtocontents{toc}{\protect\setcounter{tocdepth}{2}}

{\centering\Large\bfseries
TensorSpectra: Bridging the Spectral and Tensor Worlds\\
for High-Dimensional PDEs -- Appendix\par}
\label{sec:appendix}

\noindent\rule{\textwidth}{0.6pt}
\begingroup
\renewcommand{\contentsname}{{\Large Contents}}
\setcounter{tocdepth}{-1}
\let\appendixtocaddvspace\addvspace
\renewcommand{\addvspace}[1]{\appendixtocaddvspace{0.5em}}
\tableofcontents
\endgroup
\noindent\rule{\textwidth}{0.6pt}

\section{Spectral Basis Definitions} \label{sec:spectral basis}
In this section, we provide the detailed definitions and basic properties of the spectral bases used in TensorSpectra. Specifically, we summarize Fourier basis for periodic problems and Jacobi polynomial bases for non-periodic settings, together with the normalization, orthogonality, and coefficient conventions adopted throughout the paper. We also clarify the relationships among these polynomial bases, with Legendre and Chebyshev polynomials viewed as important special cases within the broader Jacobi family. For time-dependent problems, the spatial factors are independent of time and each CP component is multiplied by the temporal factor $T_\alpha(t)$ defined in Equation~\eqref{eq:ts-time-basis}. Stationary problems are recovered by setting $T_\alpha\equiv1$.

\subsection{Fourier Basis}
For periodic problems, TensorSpectra adopts the complex Fourier basis. On the canonical interval $[0,2\pi]$, the one-dimensional basis functions are
\begin{equation}
  \phi_k(x)=e^{\mathrm{i}kx},
  \quad k\in\mathbb Z,
\end{equation}
with $\mathrm{i}=\sqrt{-1}$. In TensorSpectra, the $j$-th spectral factor of the $\alpha$-th CP component is represented by the coefficient vector
\begin{equation}
  \mathbf a_\alpha^{(j)} = \left[a_{\alpha,k_j}^{(j)}\right]_{k_j\in\mathcal K_j},
\end{equation}
which defines the corresponding one-dimensional Fourier expansion
\begin{equation}
  \psi_\alpha^{(j)}(x_j) = \sum_{k_j\in\mathcal K_j} a_{\alpha,k_j}^{(j)} e^{\mathrm{i}k_jx_j}.
\end{equation}
Combining the one-dimensional factors across dimensions gives the separable representation
\begin{equation}
  u(\boldsymbol{x},t) = \sum_{\alpha=1}^{r} \lambda_\alpha T_\alpha(t) \prod_{j=1}^{d} \psi_\alpha^{(j)}(x_j),
\end{equation}
which is equivalent to the CP representation of the multidimensional Fourier coefficient tensor. In particular, the coefficient associated with the multidimensional mode $\boldsymbol{k}=(k_1,\ldots,k_d)$ is
\begin{equation}
  \widehat u_{\boldsymbol{k}}(t) = \sum_{\alpha=1}^{r} \lambda_\alpha T_\alpha(t) \prod_{j=1}^{d} a_{\alpha,k_j}^{(j)}.
\end{equation}

\subsection{Jacobi Basis} \label{sec:Jacobi basis}
For non-periodic problems, TensorSpectra employs Jacobi polynomials as a general polynomial spectral basis on the canonical interval $[-1,1]$. The degree-$n$ Jacobi polynomial is denoted by $J_n^{(\alpha,\beta)}(x)$, where $\alpha,\beta>-1$ determine the associated weight function
\begin{equation}
  w^{(\alpha,\beta)}(x) = (1-x)^{\alpha}(1+x)^{\beta},
  \quad x\in[-1,1].
\end{equation}
The Jacobi polynomials form an orthogonal basis with respect to this weighted inner product,
\begin{equation}
  \int_{-1}^{1} J_m^{(\alpha,\beta)}(x) J_n^{(\alpha,\beta)}(x) w^{(\alpha,\beta)}(x)\,dx = h_n^{(\alpha,\beta)}\delta_{mn},
\end{equation}
where $h_n^{(\alpha,\beta)}$ denotes the squared norm of the $n$-th basis function and $\delta_{mn}$ is the Kronecker delta.

Within TensorSpectra, the $j$-th spectral factor of the $\ell$-th CP component is represented by the coefficient vector
\begin{equation}
  \mathbf a_{\ell}^{(j)} = \left[ a_{\ell,k_j}^{(j)} \right]_{k_j=0}^{N_j-1},
\end{equation}
which defines the corresponding one-dimensional Jacobi expansion
\begin{equation}
  \psi_{\ell}^{(j)}(x_j) = \sum_{k_j=0}^{N_j-1} a_{\ell,k_j}^{(j)} J_{k_j}^{(\alpha,\beta)}(x_j).
\end{equation}
Combining these one-dimensional expansions across coordinates gives the multidimensional separable representation
\begin{equation}
  u(\boldsymbol{x},t) = \sum_{\ell=1}^{r} \lambda_{\ell} T_\ell(t) \prod_{j=1}^{d} \psi_{\ell}^{(j)}(x_j),
\end{equation}
and the corresponding Jacobi coefficient associated with the multidimensional index $\boldsymbol{k}=(k_1,\ldots,k_d)$ is
\begin{equation}
  \widehat u_{\boldsymbol{k}}(t) = \sum_{\ell=1}^{r} \lambda_{\ell} T_\ell(t) \prod_{j=1}^{d} a_{\ell,k_j}^{(j)}.
\end{equation}
Thus, the same CP coefficient structure used in the Fourier setting is retained, while only the underlying one-dimensional spectral basis is changed.

\paragraph{Legendre basis.}
The Legendre basis is obtained from the Jacobi family by setting $\alpha=\beta=0$. The corresponding weight function is therefore uniform, and the degree-$n$ Legendre polynomial satisfies
\begin{equation}
  L_n(x)=J_n^{(0,0)}(x),
  \quad w(x)=1.
\end{equation}
The Legendre polynomials are orthogonal on $[-1,1]$ under the standard $L^2$ inner product,
\begin{equation}
  \int_{-1}^{1} L_m(x)L_n(x)\,dx = \frac{2}{2n+1}\delta_{mn}.
\end{equation}
Accordingly, when TensorSpectra uses a Legendre backend, the one-dimensional factor associated with the $j$-th coordinate is expanded as
\begin{equation}
  \psi_{\ell}^{(j)}(x_j) = \sum_{k_j=0}^{N_j-1} a_{\ell,k_j}^{(j)} L_{k_j}(x_j).
\end{equation}
Thus, the CP structure remains unchanged, while the underlying one-dimensional spectral factors are interpreted as Legendre coefficients.

\paragraph{Chebyshev basis.}
The Chebyshev polynomials of the first kind provide another important special case of the Jacobi family. They are associated with $\alpha=\beta=-1/2$ and can be defined by
\begin{equation}
  T_n(x) = \cos\!\left(n\arccos x\right),
  \quad x\in[-1,1].
\end{equation}
Under the standard Jacobi normalization, the relation between Chebyshev and Jacobi polynomials is
\begin{equation}
  T_n(x) = \frac{n!\sqrt{\pi}} {\Gamma\!\left(n+\frac{1}{2}\right)} J_n^{(-1/2,-1/2)}(x).
\end{equation}
The corresponding weight function is $w(x)=(1-x^2)^{-1/2}$, and the orthogonality relation is
\begin{equation}
  \int_{-1}^{1} \frac{T_m(x)T_n(x)} {\sqrt{1-x^2}}\,dx =
  \begin{cases}
    \pi, & m=n=0,\\[2mm]
    \dfrac{\pi}{2}, & m=n\geq 1,\\[2mm]
    0, & m\neq n.
  \end{cases}
\end{equation}
When Chebyshev polynomials are used in TensorSpectra, the corresponding one-dimensional factor is represented as
\begin{equation}
  \psi_{\ell}^{(j)}(x_j) = \sum_{k_j=0}^{N_j-1} a_{\ell,k_j}^{(j)} T_{k_j}(x_j).
\end{equation}
This basis is particularly convenient for representing non-periodic coordinates and, in our implementation, temporal dependence when Chebyshev-in-time discretization is employed.

\section{De-Aliased Nonlinear Spectral Calculus}
\label{sec:de-aliased}
In this section, we provide the detailed de-aliasing procedures used by TensorSpectra to evaluate nonlinear spectral interactions while preserving the low-rank tensor structure, covering both the Fourier and Jacobi spectral backends.

\subsection{Fourier Backend}
For the Fourier backend, nonlinear interactions are evaluated using a standard padding-and-truncation procedure. Let $\mathcal{P}_{N\rightarrow M}$ denote zero-padding from the retained spectral resolution $N$ to an enlarged resolution $M>N$, and let $\mathcal{T}_{M\rightarrow N}$ denote truncation back to the retained frequency range. For each CP factor, the padded spectral coefficients are defined as
\begin{equation}
  \widetilde{\mathbf a}_{\alpha}^{(j)} = \mathcal{P}_{N\rightarrow M} \mathbf a_{\alpha}^{(j)} .
\end{equation}
The corresponding one-dimensional physical-space factor on the enlarged grid is then obtained through the inverse Fourier transform,
\begin{equation}
  \widetilde{\psi}_{\alpha}^{(j)} = \mathcal{F}_{M}^{-1} \widetilde{\mathbf a}_{\alpha}^{(j)} .
\end{equation}
Combining the transformed factors gives the padded physical-space approximation
\begin{equation}
  u_M(\boldsymbol{x}) = \sum_{\alpha=1}^{r} \lambda_{\alpha} \prod_{j=1}^{d} \widetilde{\psi}_{\alpha}^{(j)}(x_j).
\end{equation}
The nonlinear operator is evaluated on this enlarged physical grid and then transformed back to spectral space. The resulting nonlinear spectral term is truncated to the original retained bandwidth,
\begin{equation}
  \widehat{\mathcal N}(u) = \mathcal{T}_{M\rightarrow N} \mathcal{F}_{M} \left[ \mathcal N \left( \mathcal{F}_{M}^{-1} \mathcal{P}_{N\rightarrow M} \widehat{\mathcal U} \right) \right].
\end{equation}
This enlarged-grid evaluation prevents unresolved nonlinear interactions from folding back into the retained modes. For quadratic nonlinearities, TensorSpectra adopts the standard $3/2$-padding rule,
\begin{equation}
  M \simeq \frac{3}{2}N.
\end{equation}
More generally, for a polynomial nonlinearity of degree $s$, an alias-free choice under the usual symmetric Fourier convention satisfies
\begin{equation}
  M \geq \frac{s+1}{2}N,
\end{equation}
which reduces to the $3/2$ rule when $s=2$.

Importantly, the padding and inverse transforms are applied to the one-dimensional spectral factors rather than to an explicitly assembled $N^d$ coefficient tensor. This allows the nonlinear term to be evaluated in a de-aliased manner while retaining the separable structure underlying TensorSpectra.

\subsection{Jacobi Backend}
For the Jacobi backend, nonlinear interactions are de-aliased through over-integration rather than Fourier zero-padding. Let $J_n^{(\alpha,\beta)}(x)$ denote the degree-$n$ Jacobi polynomial on $[-1,1]$, orthogonal with respect to the weight $w^{(\alpha,\beta)}(x)=(1-x)^{\alpha}(1+x)^{\beta}$. For a one-dimensional nonlinear factor $g(x)$, its projection onto the retained polynomial space of degree $p$ is defined by
\begin{equation}
  \left[\Pi_{p}^{(\alpha,\beta)} g\right]_n = \frac{1}{h_n^{(\alpha,\beta)}} \int_{-1}^{1} g(x)\, J_n^{(\alpha,\beta)}(x)\, w^{(\alpha,\beta)}(x)\,dx,
  \quad 0\le n\le p,
\end{equation}
where $h_n^{(\alpha,\beta)}$ denotes the squared norm of $J_n^{(\alpha,\beta)}$. In practice, this projection is evaluated using a $Q$-point Gauss--Jacobi quadrature rule with nodes $\{x_q\}_{q=1}^{Q}$ and weights $\{\omega_q^{(\alpha,\beta)}\}_{q=1}^{Q}$,
\begin{equation}
  \left[\Pi_{p,Q}^{(\alpha,\beta)} g\right]_n = \frac{1}{h_n^{(\alpha,\beta)}} \sum_{q=1}^{Q} \omega_q^{(\alpha,\beta)} g(x_q) J_n^{(\alpha,\beta)}(x_q).
\end{equation}
Thus, nonlinear terms are first evaluated at an over-integrated set of Gauss--Jacobi nodes and then projected back to the retained degree-$p$ polynomial space. This procedure plays the same role as padding and truncation in the Fourier backend, but is adapted to the polynomial basis and its weighted inner product.

The quadrature order $Q$ is chosen according to polynomial exactness. A $Q$-point Gauss--Jacobi rule integrates weighted polynomials of degree up to $2Q-1$ exactly. If the retained approximation has degree at most $p$ and the nonlinearity is a degree-$s$ polynomial, then the projected integrand $g(x)J_n^{(\alpha,\beta)}(x)$ has degree at most $(s+1)p$. An alias-free choice therefore satisfies
\begin{equation}
  Q \ge \left\lceil \frac{(s+1)p+1}{2} \right\rceil.
\end{equation}
For quadratic nonlinearities, this reduces to
\begin{equation}
  Q \ge \left\lceil \frac{3p+1}{2} \right\rceil.
\end{equation}
Within TensorSpectra, the nonlinear products are formed from one-dimensional factors and projected independently along each coordinate. The resulting projected factors are then recombined through their outer products, so the nonlinear evaluation can be carried out without explicitly forming the full $N^d$ coefficient tensor. For non-polynomial nonlinearities, exact polynomial degree counting is unavailable; in that case, $Q$ is increased until the projected residual and target quantities are numerically converged.

\subsection{Low-Rank Nonlinear Interaction and Rank Control}
Although the Fourier and Jacobi backends use different de-aliasing mechanisms, nonlinear interactions share the same low-rank structure at the level of CP factors. Consider two rank-one components of the physical-space representation,
\begin{equation}
  u_{\alpha}(\boldsymbol{x}) = \lambda_{\alpha} \prod_{j=1}^{d} \psi_{\alpha}^{(j)}(x_j),
  \quad u_{\beta}(\boldsymbol{x}) = \lambda_{\beta} \prod_{j=1}^{d} \psi_{\beta}^{(j)}(x_j).
\end{equation}
Their product remains separable across spatial coordinates,
\begin{equation}
  u_{\alpha}(\boldsymbol{x})u_{\beta}(\boldsymbol{x}) = \lambda_{\alpha}\lambda_{\beta} \prod_{j=1}^{d} \left[ \psi_{\alpha}^{(j)}(x_j) \psi_{\beta}^{(j)}(x_j) \right].
\end{equation}
Hence, nonlinear products can be evaluated through one-dimensional factor products without forming a dense multidimensional tensor. For the quadratic term $u^2$, a direct pairwise expansion would contain $r^2$ interactions. Since $u^2$ is symmetric with respect to the pair $(\alpha,\beta)$, TensorSpectra retains only the unique pairs with $\alpha\leq\beta$ and assigns multiplicity two to the off-diagonal interactions. The resulting effective rank is therefore
\begin{equation}
  r_{\mathrm{quad}} = \frac{r(r+1)}{2},
\end{equation}
rather than $r^2$. The symmetry reduction is exact. For each retained pair, the one-dimensional products are evaluated and projected independently in every coordinate, after which the resulting factors are recombined into a CP representation.

\section{PDE Benchmarks}
In this section, we provide detailed descriptions of the PDE benchmarks used to evaluate TensorSpectra. 

\subsection{High-Dimensional Linear PDE Benchmarks}
\label{sec:linear bench}
We consider the Poisson, heat and wave equations as three canonical linear PDE benchmarks. Together, they cover elliptic, parabolic and hyperbolic PDEs respectively, and therefore provide a representative testbed for evaluating the accuracy and dimensional scalability of TensorSpectra across fundamentally different classes of differential operators and solution behaviors.

\paragraph{Poisson equation.}
We first consider the $d$-dimensional Poisson equation with periodic boundary conditions as below,
\begin{equation}
  \Delta u(\boldsymbol{x}) = f(\boldsymbol{x}),
  \quad \boldsymbol{x} \in [0,2\pi]^d.
\end{equation}
The forcing term is constructed from
\begin{equation}
  f_p(\boldsymbol{x}) = \frac{1}{\pi^2} + \frac{1}{130|\mathcal{D}|} \sum_{i\in\mathcal{D}} \cosh\!\left(2(x_i-\pi)\right),
\end{equation}
where $\mathcal{D}\subseteq\{1,\ldots,d\}$ denotes the set of active dimensions. Then we retain the valid Fourier coefficients of $f_p$ and set the zero-frequency coefficient to zero to satisfy the compatibility condition of the periodic Poisson equation. The resulting forcing and reference solution are reconstructed as
\begin{equation}
  \begin{aligned}
  f(\boldsymbol{x}) &= \sum_{\mathbf{k}\in\mathcal{K}} \operatorname{Re}\left(\widehat{f}_{\mathbf{k}} e^{\mathrm{i}\mathbf{k}\cdot\boldsymbol{x}}\right),\\
  u(\boldsymbol{x}) &= \sum_{\mathbf{k}\in\mathcal{K}} \operatorname{Re}\left(\widehat{u}_{\mathbf{k}} e^{\mathrm{i}\mathbf{k}\cdot\boldsymbol{x}}\right),\\
  \widehat{u}_{\mathbf{k}} &= -\frac{\widehat{f}_{\mathbf{k}}}{\|\mathbf{k}\|_2^2}.
  \end{aligned}
\end{equation}
We evaluate all methods for $d\in\{2,5,10,30,50,100\}$, providing a controlled benchmark for examining how solution accuracy scales with dimensionality.

\paragraph{Heat equation.}
We next consider the $d$-dimensional heat equation with periodic boundary conditions as below,
\begin{equation}
  \frac{\partial u}{\partial t}(\boldsymbol{x},t) = \Delta u(\boldsymbol{x},t),
  \quad (\boldsymbol{x},t)\in[0,2\pi]^d\times[0,0.1].
\end{equation}
The initial condition and reference solution are constructed from the solution of the Poisson problem introduced above,
\begin{equation}
  \begin{aligned}
  u(\boldsymbol{x},0) &= \sum_{\mathbf{k}\in\mathcal{K}} \operatorname{Re}\left(\widehat{u}_{\mathbf{k}} e^{\mathrm{i}\mathbf{k}\cdot\boldsymbol{x}}\right),\\
  u(\boldsymbol{x},t) &= \sum_{\mathbf{k}\in\mathcal{K}} \operatorname{Re}\left(\widehat{u}_{\mathbf{k}} e^{-\|\mathbf{k}\|_2^2t} e^{\mathrm{i}\mathbf{k}\cdot\boldsymbol{x}}\right),
  \end{aligned}
\end{equation}
where $\widehat{u}_{\mathbf{k}}$ denotes the corresponding Fourier coefficient. We again consider $d\in\{2,5,10,30,50,100\}$, allowing us to evaluate the dimensional scalability of different methods on a time-dependent parabolic PDE.

\paragraph{Wave equation.}
Finally, we consider a $d$-dimensional wave equation constructed using the same spatial spectral structure as the preceding Poisson and heat benchmarks,
\begin{equation}
  \frac{\partial^2 u}{\partial t^2}(\boldsymbol{x},t) = \Delta u(\boldsymbol{x},t),
  \quad (\boldsymbol{x},t)\in[0,2\pi]^d\times[0,0.1].
\end{equation}
We impose periodic boundary conditions and use the solution of the Poisson benchmark as the initial displacement, while setting the initial velocity to zero,
\begin{equation}
  u(\boldsymbol{x},0) = \sum_{\mathbf{k}\in\mathcal{K}} \operatorname{Re} \left( \widehat{u}_{\mathbf{k}} e^{\mathrm{i}\mathbf{k}\cdot\boldsymbol{x}} \right),\quad \frac{\partial u}{\partial t}(\boldsymbol{x},0) =0.
\end{equation}
Since the Fourier modes diagonalize the Laplace operator, each spatial mode evolves independently according to a harmonic oscillator. The corresponding reference solution is therefore
\begin{equation}
  u(\boldsymbol{x},t) = \sum_{\mathbf{k}\in\mathcal{K}} \operatorname{Re} \left( \widehat{u}_{\mathbf{k}} \cos\!\left(\|\mathbf{k}\|_2 t\right) e^{\mathrm{i}\mathbf{k}\cdot\boldsymbol{x}} \right).
\end{equation}
As in the Poisson and heat benchmarks, we consider $d\in\{2,5,10,30,50,100\}$. In contrast to the dissipative dynamics of the heat equation, the wave equation preserves oscillatory modes over time, providing a complementary test of TensorSpectra on high-dimensional hyperbolic dynamics.

\subsection{Middle-Dimensional Nonlinear PDE Benchmarks}
\label{sec:nonlinear bench}
We further evaluate TensorSpectra on two middle-dimensional nonlinear PDE benchmarks, including a steady-state Allen--Cahn equation and a nonlinear reaction equation. These problems introduce nonlinear interactions that destroy the simple mode-wise separability available in linear PDEs, providing a more challenging setting for assessing the robustness of the proposed low-rank spectral formulation. To further demonstrate the spectral extensibility of TensorSpectra beyond Fourier representations, we employ Legendre polynomial basis mentioned in Appendix~\ref{sec:Jacobi basis} for all nonlinear benchmarks in this section.

\paragraph{Steady-state Allen--Cahn equation.}
We consider the $d$-dimensional steady-state Allen--Cahn equation on $\Omega = [-1,1]^d$,
\begin{equation}
  -\varepsilon^2\Delta u(\boldsymbol{x}) +u^3(\boldsymbol{x})-u(\boldsymbol{x})= f(\boldsymbol{x}), \quad \boldsymbol{x}\in \Omega,
\end{equation}
subject to the homogeneous Dirichlet boundary condition
\begin{equation}
  u(\boldsymbol{x})=0, \quad \boldsymbol{x}\in\partial\Omega.
\end{equation}
We adopt the manufactured solution
\begin{equation}
  u(\boldsymbol{x}) = \prod_{j=1}^{d}(1-x_j^2),
\end{equation}
from which the forcing term is obtained analytically as
\begin{equation}
  f(\boldsymbol{x}) = 2\varepsilon^2 \sum_{j=1}^{d} \prod_{k\neq j}(1-x_k^2) + \prod_{j=1}^{d}(1-x_j^2)^3 - \prod_{j=1}^{d}(1-x_j^2).
\end{equation}
We set $\varepsilon=0.5$ and consider middle-dimensional cases. In contrast to the Fourier-based linear benchmarks, the solution is represented using an orthonormal Legendre basis, while the Dirichlet condition is imposed exactly through the factor-wise construction $(1-x_j^2)q_j(x_j)$. The cubic reaction term generates nonlinear interactions among spectral modes, making this benchmark a direct test of TensorSpectra's ability to perform dealiased nonlinear operations while preserving a low-rank spectral representation.

\paragraph{Nonlinear reaction equation.}
We further consider the logistic reaction equation
\begin{equation}
  \begin{aligned}
  \frac{\partial u}{\partial t}(\boldsymbol{x},t) - \rho\,u(\boldsymbol{x},t) \left(1-u(\boldsymbol{x},t)\right) &= 0,\\
  (\boldsymbol{x},t) &\in [0,2\pi]^d\times[0,1],
  \end{aligned}
\end{equation}
with the Gaussian initial condition
\begin{equation}
  u(\boldsymbol{x},0) = h(\boldsymbol{x}) = \exp\!\left( -\frac{ \|\boldsymbol{x}-\pi\boldsymbol{1}\|_2^2 }{ 2(\pi/4)^2 } \right).
\end{equation}
The initial profile is separable across spatial dimensions,
\begin{equation}
  h(\boldsymbol{x}) = \prod_{j=1}^{d} \exp\!\left( -\frac{ (x_j-\pi)^2 }{ 2(\pi/4)^2 } \right),
\end{equation}
and the corresponding exact solution is
\begin{equation}
  u(\boldsymbol{x},t) = \frac{ h(\boldsymbol{x})e^{\rho t} }{ h(\boldsymbol{x})e^{\rho t} + 1-h(\boldsymbol{x}) }.
\end{equation}
To evaluate TensorSpectra beyond Fourier representations, both the spatial coordinates and time are expanded in shifted orthonormal Legendre bases, with time treated as an additional tensor dimension. The Gaussian initial condition is deterministically projected onto the Legendre basis and imposed through the hard construction $u(\boldsymbol{x},t)=h_K(\boldsymbol{x})+t\,v(\boldsymbol{x},t)$, while the quadratic reaction term is evaluated through dealiased low-rank spectral multiplication. This benchmark therefore tests whether TensorSpectra can preserve a compact spectral representation under nonlinear temporal dynamics without relying on physical-space collocation.

\section{PDE-Specific Spectral Objectives}
\label{sec:pde-loss}
In this section, we provide the PDE-specific instantiations of the spectral training objective introduced in the main text. For each benchmark, we derive the corresponding spectral residual, identify the characteristic operator scale used to construct the weighting $w_{\boldsymbol{k}}$, and state the resulting operator-aware objective. We consider the Poisson, Heat, Wave, steady-state Allen--Cahn and nonlinear Reaction equations used in our experiments.

\subsection{Poisson Equation}
We first consider the periodic Poisson equation used in our linear benchmarks,
\begin{equation}
  \Delta u(\boldsymbol{x}) = f(\boldsymbol{x}),
  \quad \boldsymbol{x}\in[0,2\pi]^d .
\end{equation}
Under the Fourier representation, the Laplace operator acts diagonally on each spectral mode, with $\widehat{\Delta u}_{\boldsymbol{k}}=-\|\boldsymbol{k}\|_2^2\widehat{u}_{\boldsymbol{k}}$. Therefore, the spectral residual associated with mode $\boldsymbol{k}$ is
\begin{equation}
  \mathcal{R}_{\boldsymbol{k}} = -\|\boldsymbol{k}\|_2^2 \widehat{u}_{\boldsymbol{k}} - \widehat{f}_{\boldsymbol{k}} .
\end{equation}
The characteristic spectral scale of the Poisson operator is thus determined by the Laplacian symbol,
\begin{equation}
  \sigma_{\Delta}(\boldsymbol{k}) = \|\boldsymbol{k}\|_2^2,
  \quad w_{\boldsymbol{k}} = \rho\!\left( \|\boldsymbol{k}\|_2^2 \right),
\end{equation}
where $\rho(\cdot)$ denotes the decreasing scaling function introduced in the main text.

Accordingly, TensorSpectra minimizes the operator-aware spectral objective
\begin{equation}
  \mathcal{L}_{\mathrm{Poisson}} = \frac{1}{|\mathcal{K}_{\!*}|} \sum_{\boldsymbol{k}\in\mathcal{K}_{\!*}} w_{\boldsymbol{k}}^2 \left| -\|\boldsymbol{k}\|_2^2 \widehat{u}_{\boldsymbol{k}} - \widehat{f}_{\boldsymbol{k}} \right|^2 ,
\end{equation}
where $\mathcal{K}_{\!*}$ denotes the retained nonzero Fourier modes. The zero-frequency mode is excluded consistently with the compatibility condition imposed in the periodic Poisson benchmark.

\subsection{Heat Equation}
We next consider the $d$-dimensional Heat equation with periodic boundary conditions,
\begin{equation}
  \frac{\partial u}{\partial t}(\boldsymbol{x},t) = \Delta u(\boldsymbol{x},t),
  \quad (\boldsymbol{x},t)\in[0,2\pi]^d\times[0,T].
\end{equation}
Under the Fourier representation, the Laplace operator satisfies $\widehat{\Delta u}_{\boldsymbol{k}}=-\|\boldsymbol{k}\|_2^2\widehat{u}_{\boldsymbol{k}}$. Hence, the coefficient-space evolution equation for each Fourier mode is
\begin{equation}
  \partial_t\widehat{u}_{\boldsymbol{k}}(t) + \|\boldsymbol{k}\|_2^2 \widehat{u}_{\boldsymbol{k}}(t) = 0,
\end{equation}
and the corresponding spectral residual is
\begin{equation}
  \mathcal{R}_{\boldsymbol{k}}(t) = \partial_t\widehat{u}_{\boldsymbol{k}}(t) + \|\boldsymbol{k}\|_2^2 \widehat{u}_{\boldsymbol{k}}(t).
\end{equation}
Following the operator-aware construction in the main text, the spatial operator scale is determined by the Laplacian symbol,
\begin{equation}
  \sigma_{\mathrm{Heat}}(\boldsymbol{k}) = \|\boldsymbol{k}\|_2^2,
  \quad w_{\boldsymbol{k}} = \rho\!\left( \|\boldsymbol{k}\|_2^2 \right).
\end{equation}
The resulting operator-aware PDE objective is therefore
\begin{equation}
  \mathcal{L}_{\mathrm{Heat}}^{\mathrm{PDE}} = \frac{1}{|\mathcal{S}|} \sum_{(\boldsymbol{k},t)\in\mathcal{S}} w_{\boldsymbol{k}}^2 \left| \partial_t\widehat{u}_{\boldsymbol{k}}(t) + \|\boldsymbol{k}\|_2^2 \widehat{u}_{\boldsymbol{k}}(t) \right|^2.
\end{equation}
The initial condition is imposed in spectral space by matching the predicted coefficients at $t=0$ to the prescribed initial coefficients $\widehat{u}_{\boldsymbol{k}}^{\,0}$,
\begin{equation}
  \mathcal{L}_{\mathrm{IC}} = \frac{1}{|\mathcal{K}|} \sum_{\boldsymbol{k}\in\mathcal{K}} \left| \widehat{u}_{\boldsymbol{k}}(0) - \widehat{u}_{\boldsymbol{k}}^{\,0} \right|^2,
\end{equation}
leading to the overall Heat-equation objective
\begin{equation}
  \mathcal{L}_{\mathrm{Heat}} = \mathcal{L}_{\mathrm{Heat}}^{\mathrm{PDE}} + \lambda_{\mathrm{IC}}\mathcal{L}_{\mathrm{IC}}.
\end{equation}
In the benchmark used in this work, the initial spectral coefficients are inherited from the Poisson solution introduced above, and each Fourier mode subsequently evolves independently under the Heat operator.

\subsection{Wave Equation}
We next consider the $d$-dimensional Wave equation with periodic boundary conditions,
\begin{equation}
  \frac{\partial^2 u}{\partial t^2}(\boldsymbol{x},t) = \Delta u(\boldsymbol{x},t),
  \quad (\boldsymbol{x},t)\in[0,2\pi]^d\times[0,T].
\end{equation}
Under the Fourier representation, the Laplace operator acts diagonally on each spatial mode. Hence, every Fourier coefficient satisfies the decoupled second-order evolution equation
\begin{equation}
  \partial_{tt}\widehat{u}_{\boldsymbol{k}}(t) + \|\boldsymbol{k}\|_2^2 \widehat{u}_{\boldsymbol{k}}(t) = 0,
\end{equation}
and the corresponding spectral residual is
\begin{equation}
  \mathcal{R}_{\boldsymbol{k}}(t) = \partial_{tt}\widehat{u}_{\boldsymbol{k}}(t) + \|\boldsymbol{k}\|_2^2 \widehat{u}_{\boldsymbol{k}}(t).
\end{equation}
Following the operator-aware construction, the characteristic spatial scale is determined by the Fourier symbol of the Laplacian,
\begin{equation}
  \sigma_{\mathrm{Wave}}(\boldsymbol{k}) = \|\boldsymbol{k}\|_2^2,
  \quad w_{\boldsymbol{k}} = \rho\!\left( \|\boldsymbol{k}\|_2^2 \right).
\end{equation}
The resulting operator-aware PDE objective is
\begin{equation}
  \mathcal{L}_{\mathrm{Wave}}^{\mathrm{PDE}} = \frac{1}{|\mathcal{S}|} \sum_{(\boldsymbol{k},t)\in\mathcal{S}} w_{\boldsymbol{k}}^2 \left| \partial_{tt}\widehat{u}_{\boldsymbol{k}}(t) + \|\boldsymbol{k}\|_2^2 \widehat{u}_{\boldsymbol{k}}(t) \right|^2.
\end{equation}
Because the Wave equation is second order in time, both the initial displacement and initial velocity must be prescribed. In our benchmark, the initial displacement is inherited from the Poisson solution and the initial velocity is set to zero. Accordingly, we define
\begin{equation}
  \mathcal{L}_{\mathrm{IC}} = \frac{1}{|\mathcal{K}|} \sum_{\boldsymbol{k}\in\mathcal{K}} \left| \widehat{u}_{\boldsymbol{k}}(0) - \widehat{u}_{\boldsymbol{k}}^{\,0} \right|^2 + \frac{1}{|\mathcal{K}|} \sum_{\boldsymbol{k}\in\mathcal{K}} \left| \partial_t\widehat{u}_{\boldsymbol{k}}(0) - \widehat{v}_{\boldsymbol{k}}^{\,0} \right|^2,
\end{equation}
where $\widehat{v}_{\boldsymbol{k}}^{\,0}=0$ in the present benchmark. The overall training objective is therefore
\begin{equation}
  \mathcal{L}_{\mathrm{Wave}} = \mathcal{L}_{\mathrm{Wave}}^{\mathrm{PDE}} + \lambda_{\mathrm{IC}} \mathcal{L}_{\mathrm{IC}}.
\end{equation}

\subsection{Steady-State Allen--Cahn Equation}
We next consider the steady-state Allen--Cahn equation on the non-periodic domain $[-1,1]^d$,
\begin{equation}
  -\epsilon^2\Delta u(\boldsymbol{x}) + u^3(\boldsymbol{x}) - u(\boldsymbol{x}) = f(\boldsymbol{x}),
  \quad \boldsymbol{x}\in[-1,1]^d,
\end{equation}
subject to the homogeneous Dirichlet boundary condition $u(\boldsymbol{x})=0$ on $\partial[-1,1]^d$. In this benchmark, TensorSpectra employs an orthonormal Legendre representation, and the boundary condition is imposed exactly through the factor-wise construction.

In coefficient space, the corresponding spectral residual is
\begin{equation}
  \widehat{\mathcal R}_{\mathrm{AC}} = -\epsilon^2\widehat{\Delta u} + \widehat{u^3} - \widehat{\mathcal U} - \widehat{\mathcal F},
\end{equation}
where $\widehat{u^3}$ is evaluated using the de-aliased low-rank nonlinear spectral calculus described in the main text and appendix.

Equivalently, for the retained multidimensional spectral index $\boldsymbol{k}$,
\begin{equation}
  \mathcal R_{\boldsymbol{k}}^{\mathrm{AC}} = -\epsilon^2 \left[\widehat{\Delta u}\right]_{\boldsymbol{k}} + \left[\widehat{u^3}\right]_{\boldsymbol{k}} - \widehat u_{\boldsymbol{k}} - \widehat f_{\boldsymbol{k}}.
\end{equation}
Unlike the Fourier cases, the Legendre differentiation operator does not admit a diagonal mode-wise symbol. We therefore construct the operator scale directly from the one-dimensional coefficient-space differentiation matrix. Let $\mathbf D^{(2)}$ denote the second-order Legendre differentiation matrix and $\mathbf e_n$ the $n$-th canonical coefficient vector. We define $\eta_n=\|\mathbf D^{(2)}\mathbf e_n\|_2$. For the multidimensional mode $\boldsymbol{k}=(k_1,\ldots,k_d)$, we then define the characteristic scale of the stiff linear part $-\varepsilon^2\Delta-I$ by
\begin{equation}
  \sigma_{\mathrm{AC}}(\boldsymbol{k}) = \left( 1 + \varepsilon^4 \sum_{j=1}^{d} \eta_{k_j}^{\,2} \right)^{1/2}.
\end{equation}
The corresponding operator-aware spectral weight is chosen as an inverse power of this scale,
\begin{equation}
  w_{\boldsymbol{k}} = \sigma_{\mathrm{AC}}(\boldsymbol{k})^{-p} = \left( 1 + \varepsilon^4 \sum_{j=1}^{d} \eta_{k_j}^{\,2} \right)^{-p/2},
  \quad 0<p\leq 1.
\end{equation}
The weight is constructed from the linear operator, while the full residual includes the cubic term. The resulting PDE-aware objective is
\begin{equation}
  \mathcal L_{\mathrm{AC}} = \frac{1}{|\mathcal K|} \sum_{\boldsymbol{k}\in\mathcal K} w_{\boldsymbol{k}}^2 \left| \mathcal R_{\boldsymbol{k}}^{\mathrm{AC}} \right|^2.
\end{equation}
Because the homogeneous Dirichlet boundary condition is enforced directly through the factor-wise parameterization, no additional boundary penalty is required in this benchmark.

\subsection{Nonlinear Reaction Equation}
We finally consider the nonlinear logistic reaction equation
\begin{equation}
  \frac{\partial u}{\partial t}(\boldsymbol{x},t) - \rho\,u(\boldsymbol{x},t) \left( 1-u(\boldsymbol{x},t) \right) = 0,
  \quad (\boldsymbol{x},t)\in[0,2\pi]^d\times[0,T],
\end{equation}
with a prescribed initial condition $u(\boldsymbol{x},0)=h(\boldsymbol{x})$. In this benchmark, both the spatial coordinates and time are represented using shifted orthonormal Legendre bases, with time treated as an additional tensor dimension. The projected initial condition $h_K$ is imposed exactly through the hard construction $u(\boldsymbol{x},t)=h_K(\boldsymbol{x})+t\,v(\boldsymbol{x},t)$, so that no additional initial-condition penalty is required.

The hard initial-condition parameterization is
\begin{equation}
  u(\boldsymbol{x},t) = h_K(\boldsymbol{x}) + t\,v(\boldsymbol{x},t).
\end{equation}
Expanding the nonlinear term gives $\partial_tu-\rho u+\rho u^2=0$. Accordingly, the coefficient-space residual takes the form
\begin{equation}
  \widehat{\mathcal R}_{\mathrm{React}} = \partial_t \widehat{u} - \rho\,\widehat{\mathcal U} + \rho\,\widehat{u^2},
\end{equation}
where $\widehat{u^2}$ is computed through the de-aliased low-rank Legendre product introduced previously.  Let $\boldsymbol{k}=(k_1,\ldots,k_d)$ denote the spatial spectral index and $q$ the temporal Legendre index. The corresponding space--time coefficient residual is
\begin{equation}
  \mathcal R_{\boldsymbol{k},q}^{\mathrm{React}} = \left[\partial_t \widehat{u}\right]_{\boldsymbol{k},q} - \rho\,\widehat u_{\boldsymbol{k},q} + \rho\left[\widehat{u^2}\right]_{\boldsymbol{k},q}.
\end{equation}
To construct an explicit operator-aware scale, we consider the linear part of the reaction operator, $\partial_t-\rho I$. Let $\mathbf D_t$ denote the one-dimensional differentiation matrix associated with the shifted orthonormal Legendre basis in time, and let $\mathbf e_q$ denote the canonical coefficient vector corresponding to temporal degree $q$. We define the temporal differentiation scale by
\begin{equation}
  \eta_q = \left\| \mathbf D_t\mathbf e_q \right\|_2.
\end{equation}
For the shifted orthonormal Legendre basis on $[0,T]$,
\begin{equation}
  \psi_q^{(t)}(t) = \sqrt{\frac{2q+1}{T}}\, L_q\!\left(\frac{2t}{T}-1\right),
\end{equation}
the Euclidean norm of $\mathbf D_t\mathbf e_q$ equals the $L^2(0,T)$ norm of the derivative because the basis is orthonormal. Using $\int_{-1}^{1}[L_q'(\xi)]^2d\xi=q(q+1)$, the scale can be evaluated explicitly as
\begin{equation}
  \eta_q^2 = \left\| \partial_t\psi_q^{(t)} \right\|_{L^2(0,T)}^2 = \frac{ 2q(q+1)(2q+1) }{ T^2 }.
\end{equation}

Because differentiation lowers the Legendre polynomial degree, $\mathbf D_t\mathbf e_q$ has no component along $\mathbf e_q$. The derivative contribution and the reaction contribution are therefore orthogonal in the coefficient-space norm. Hence, the characteristic scale of the linear operator $\partial_t-\rho I$ for temporal mode $q$ is
\begin{equation}
  \sigma_{\mathrm{React}}(q) = \left\| \left( \mathbf D_t-\rho\mathbf I \right) \mathbf e_q \right\|_2 = \left( \rho^2+\eta_q^2 \right)^{1/2}.
\end{equation}
Substituting the explicit Legendre differentiation scale gives
\begin{equation}
  \sigma_{\mathrm{React}}(q) = \left[ \rho^2 + \frac{ 2q(q+1)(2q+1) }{ T^2 } \right]^{1/2}.
\end{equation}
We normalize this scale by its lowest-order value $\rho$, so that the constant temporal mode receives unit weight. The resulting operator-aware weight is
\begin{equation}
  w_{\boldsymbol{k},q} = w_q = \left( \frac{\sigma_{\mathrm{React}}(q)}{\rho} \right)^{-p} = \left[ 1+ \frac{ 2q(q+1)(2q+1) }{ \rho^2T^2 } \right]^{-p/2},
  \quad 0<p\leq1.
\end{equation}
Since the reaction equation contains no spatial differential operator, the preconditioning weight depends only on the temporal spectral index $q$ and is shared by all spatial modes with the same temporal degree. The weight is constructed from the linear temporal operator, while the full residual includes the quadratic term $\rho u^2$. The resulting PDE-aware spectral objective is
\begin{equation}
  \mathcal L_{\mathrm{React}} = \frac{1}{|\mathcal K||\mathcal Q|} \sum_{\boldsymbol{k}\in\mathcal K} \sum_{q\in\mathcal Q} w_q^2 \left| \mathcal R_{\boldsymbol{k},q}^{\mathrm{React}} \right|^2,
\end{equation}
where $\mathcal K$ and $\mathcal Q$ denote the retained spatial and temporal spectral index sets, respectively. Since the initial condition is imposed exactly through the hard parameterization, no separate initial-condition loss is required.

\section{Baseline Methods} \label{sec:baseline methods}
In this section, we summarize the main ideas of the baseline methods used for comparison.

\paragraph{Physics-Informed Neural Network (PINN).}
PINN approximates the PDE solution with a neural network and trains it by minimizing the governing-equation residual together with boundary- and initial-condition errors. The required derivatives are computed through automatic differentiation at collocation points, providing a general physics-informed framework for both stationary and time-dependent problems.

\paragraph{Separable Physics-Informed Neural Network (SPINN).}
SPINN represents the solution as a low-rank sum of products of one-dimensional neural-network outputs, with a separate subnetwork for each coordinate. This separable architecture reduces network evaluations on tensor-product collocation grids and enables efficient forward-mode automatic differentiation, while retaining a physics-informed objective based on PDE residuals and auxiliary constraints.

\paragraph{Physics-Informed Residual Adaptive Network (PirateNet).}
PirateNet improves the trainability of deep PINNs through random Fourier embeddings, gated feature transformations, and adaptive residual connections. Its residual weights are initialized to zero so that the network starts from a shallow representation and learns to activate nonlinear blocks during training. The model retains the physics-informed objective while using this architecture and initialization to facilitate optimization of PDE residuals involving higher-order derivatives.

\paragraph{Spectral-Informed Neural Network (SINN).}
SINN learns spectral coefficients from frequency indices and, for time-dependent problems, time, rather than predicting solution values directly from spatial coordinates. It enforces the governing equation in the spectral domain, where spatial differentiation becomes algebraic operations on coefficients. For high-dimensional problems, we use the modified SINN variant with coefficient-decay scaling and learnable basis embeddings to incorporate prior spectral structure into the representation.

\paragraph{Low-Rank Fourier Sums (LoRFS).}
LoRFS represents the solution as a low-rank sum of products of finite one-dimensional Fourier expansions and directly optimizes their coefficients. It combines this representation with continuous integral-based physics objectives, exploiting Fourier orthogonality and separability to evaluate residual integrals deterministically without high-dimensional stochastic collocation. Its central idea is to pair an explicit low-rank spectral representation with efficient evaluation of the physics-informed loss.

\section{Theoretical Analysis and Proofs}\label{app:ts-proofs}
This section develops the analysis underlying TensorSpectra in the order of its construction. We first fix the coefficient norms and approximation errors and quantify representation costs, then study approximation accuracy and residual stability. The final subsection analyzes low-rank corrections and establishes the de-aliasing conditions used by the nonlinear calculus. Each result is followed by its proof.

\subsection{Notation and Computational Cost}
\paragraph{Spectral spaces and error terms.}

Let $H=L^2(\Omega,\mu)$, where $\mu$ is normalized periodic measure for Fourier basis or the product Jacobi measure for polynomial basis.  We use orthonormal one-dimensional bases; unnormalized Jacobi coefficients must first be rescaled by their basis norms.  Let $V_N=\bigotimes_{j=1}^d V_{N_j}^{(j)}$ and let $P_N$ be the $H$-orthogonal projector.  Write $u_N=P_Nu$, let $\widehat{\mathcal U}_N$ be its coefficient tensor, and set
\begin{equation}\label{eq:ts-errors}
  \begin{aligned}
    E_{\mathrm{spec}}(N) &= \|u-u_N\|_H,\\
    E_{\mathrm{rank}}(N,r) &= \inf_{\operatorname{rank}_{\rm CP}(\widehat{\mathcal V})\le r} \|\widehat{\mathcal U}_N-\widehat{\mathcal V}\|_F.
  \end{aligned}
\end{equation}
Here $\bigotimes$ forms a tensor-product function space, whereas $\circ$ denotes the outer product of coefficient vectors used in the CP representation of Section~\ref{sec:ts-method}. We use an infimum because the rank-$r$ CP set need not be closed for tensor order at least three.  For a trained rank-$r$ coefficient tensor $\widehat{\mathcal U}_\theta$, choose any $\widehat{\mathcal V}_\eta$ of rank at most $r$ with $\|\widehat{\mathcal U}_N-\widehat{\mathcal V}_\eta\|_F\le E_{\mathrm{rank}}(N,r)+\eta$, and put $E_{\mathrm{opt}}(\theta;\widehat{\mathcal V}_\eta) =\|\widehat{\mathcal V}_\eta-\widehat{\mathcal U}_\theta\|_F$. This term records the gap to an approximately best solution approximation; a small PDE training loss controls it only under a separate stability and consistency argument.

\paragraph{Coefficient-to-function isometry.}

For Fourier results we equip $\mathbb T^d=[0,2\pi]^d$ with $(2\pi)^{-d}d\boldsymbol x$, so $e^{i\boldsymbol k\cdot\boldsymbol x}$ is orthonormal. For Jacobi results we use orthonormalized polynomials in $L^2_w([-1,1]^d)$, consistent with Section~\ref{sec:ts-method}. Here $\mathcal K=\prod_{j=1}^d\mathcal K_j$ is the retained multi-index set, with $|\mathcal K_j|=N_j$.  Consequently, the synthesis map $S_N$ from coefficient tensors to functions satisfies
\begin{equation}\label{eq:app-isometry}
  \|S_N\widehat{\mathcal V}\|_H^2 =\sum_{\boldsymbol k\in\mathcal K} |\widehat{\mathcal V}_{\boldsymbol k}|^2 =\|\widehat{\mathcal V}\|_F^2.
\end{equation}
In another solution norm $X$, use the finite-dimensional comparison $\|S_N\widehat{\mathcal V}\|_X \le C_{N,X}\|\widehat{\mathcal V}\|_F$; $C_{N,X}$ may depend on $N$ and $d$. Time-dependent Fourier models can be viewed as spatial CP tensors with time-varying coefficients or as space--time CP tensors after expanding the time factors.

\paragraph{Why a best CP approximation need not exist.}
The use of an infimum in $E_{\mathrm{rank}}$ is essential rather than merely notational. The following example explains why the approximation bounds use a near-best tensor $\widehat{\mathcal V}_\eta$ instead of assuming that an optimizer is attained.

\begin{proposition}[Why an infimum is necessary]\label{prop:app-nonclosed}
For order-three tensors over $\mathbb R^2$, rank at most two is not closed.
\end{proposition}
\begin{proof}
Let $e_1,e_2$ be standard basis vectors and set $v_\epsilon=e_1+\epsilon e_2$.  Using the same outer-product notation as in the CP representation, define
\begin{equation}
  T_\epsilon=\epsilon^{-1} \bigl(v_\epsilon\circ v_\epsilon\circ v_\epsilon -e_1\circ e_1\circ e_1\bigr).
\end{equation}
Each $T_\epsilon$ is a sum of two rank-one tensors. As $\epsilon\to0$ it converges to
\begin{equation}
  T=e_2\circ e_1\circ e_1 +e_1\circ e_2\circ e_1 +e_1\circ e_1\circ e_2.
\end{equation}
Its two slices along the first mode are
\begin{equation}
  A=\begin{pmatrix}0&1\\1&0\end{pmatrix},
    \quad
  B=\begin{pmatrix}1&0\\0&0\end{pmatrix}.
\end{equation}
If $T$ had rank at most two, the invertible slice $A$ would force both remaining factor matrices to be invertible and $A^{-1}B$ to be similar to a diagonal matrix.  But
\begin{equation}
  A^{-1}B=\begin{pmatrix}0&0\\1&0\end{pmatrix}
\end{equation}
is a nonzero nilpotent Jordan block.  Hence $T$ has rank three while it is a limit of rank-two tensors.
\end{proof}

The CP representation reduces storage at a prescribed rank. We first count linear factor operations, then account for the interactions created by nonlinear terms.
\paragraph{Storage and linear operators.}

\begin{proposition}[Storage and factorwise linear work]\label{prop:ts-cost}
A dense coefficient tensor for $V_N$ stores $\prod_jN_j$ scalar coefficients.  A CP parameterization with $r$ terms stores $r+r\sum_jN_j$ scalar parameters. For $N_j=N$, these are $N^d$ and $r+dNr$, respectively.  A coordinatewise linear operation on all $d$ factors costs $\mathcal{O}\bigl(r\sum_j C_{\rm sp}(N_j)\bigr)$, where $C_{\rm sp}(N_j)$ is the cost of one one-dimensional operation.  Thus Fourier differentiation costs $\mathcal{O}(dNr)$, while a naive dense Jacobi differentiation matrix costs $\mathcal{O}(dN^2r)$.  A sum of $L$ tensor-product operator terms can create up to $Lr$ CP terms; forming and contracting that result has additional cost.
\end{proposition}
\hypertarget{proof:ts-cost}{}
\begin{proof}[Proof of Proposition~\ref{prop:ts-cost}] 
There is one dense coefficient for each multi-index in $\mathcal K=\prod_{j=1}^d\mathcal K_j$, giving $\prod_jN_j$ entries.  A displayed CP representation has $r$ amplitudes and $rN_j$ entries in axis $j$, giving $r+r\sum_jN_j$.    Applying a one-dimensional map to the $r$ factor vectors in axis $j$ takes $rC_{\rm sp}(N_j)$ operations.  Summing over $j$ proves the stated bound. Fourier differentiation multiplies each coefficient by $ik_j$, so $C_{\rm sp}(N)=\mathcal{O}(N)$.  A dense matrix-vector product uses $\mathcal{O}(N^2)$ arithmetic. A tensor-product term applies one map in each relevant axis and remains one CP term per input term; a sum of $L$ such operator terms therefore has at most $Lr$ terms before cancellation. Materializing all factors in that sum can cost $\mathcal{O}(LdrN)$, even if the individual transformed vectors were cheaper to compute.
\end{proof}

\paragraph{Nonlinear interactions.}
A coordinatewise derivative preserves each CP term, whereas multiplication pairs components across the expansion. The next bound includes those interactions and factorwise projection, but separates them from the cost of contracting a residual norm. The required padding and quadrature conditions are proved in Proposition~\ref{prop:app-dealias}.

\begin{proposition}[Polynomial nonlinearities]\label{prop:ts-nonlinear}
For a quadratic nonlinearity, the symmetric expansion of a rank-$r$ function contains at most $P_2=\binom{r+1}{2}$ distinct rank-one interactions.  With $N=2K+1$ retained Fourier modes per axis and padded size $M>3K$, factorwise FFT evaluation and retained-mode truncation take $\mathcal{O}\bigl(d(r+P_2)M\log M+dP_2N\bigr)$ arithmetic.  For degree-$p$ Jacobi factors, $Q$ quadrature nodes, and dense one-dimensional evaluation and projection matrices, the corresponding bound is $\mathcal{O}\bigl(d(r+P_2)Q(p+1)\bigr)$.  The nonlinear CP output requires $\mathcal{O}(dP_2N)$ or $\mathcal{O}(dP_2(p+1))$ storage.  A cubic interaction has at most $P_3=\binom{r+2}{3}$ terms. These bounds exclude the subsequent residual-norm contraction, which can cost $\mathcal{O}(dNR^2)$ for a residual of CP rank $R$.
\end{proposition}
\hypertarget{proof:ts-nonlinear}{}
\begin{proof}[Proof of Proposition~\ref{prop:ts-nonlinear}] Writing
\begin{equation}
  u=\sum_{\alpha=1}^r\lambda_\alpha \prod_{j=1}^d\psi_\alpha^{(j)}
\end{equation}
gives
\begin{equation}\label{eq:app-square}
  u^2=\sum_{1\le\alpha\le\beta\le r} (2-\delta_{\alpha\beta})\lambda_\alpha\lambda_\beta \prod_j\bigl(\psi_\alpha^{(j)}\psi_\beta^{(j)}\bigr).
\end{equation}
There are $P_2=r(r+1)/2$ pairs.  The one-dimensional physical values of $r$ padded Fourier factors cost $\mathcal{O}(drM\log M)$.  Forming every pair and transforming it back costs $\mathcal{O}(dP_2M\log M)$, and retaining $N$ coefficients per factor costs $\mathcal{O}(dP_2N)$. The output has at most $P_2$ amplitudes and $dP_2N$ factor entries. For Jacobi, a dense $Q\times(p+1)$ evaluation matrix evaluates each factor in $\mathcal{O}(Q(p+1))$ arithmetic. Each pair needs $Q$ point products and a dense $(p+1)\times Q$ projection, hence $\mathcal{O}(Q(p+1))$ per pair and axis. Summing gives the asserted cost. Cubic terms are indexed by multisets of size three, numbering $\binom{r+2}{3}$; their multiplicities are the multinomial coefficients, and an exact intermediate polynomial product must retain sufficient degree. Finally, the inner product of CP tensors with $R$ terms is a sum over $R^2$ pairs of products of $d$ one-dimensional inner products, each of length $N$, giving $\mathcal{O}(dNR^2)$ without additional structure.
\end{proof}

\subsection{Spectral and Low-Rank Approximation}
We now separate the errors introduced by finite spectral resolution and finite CP rank. The Fourier and Jacobi estimates below use the orthonormal conventions fixed above and classical spectral approximation tools \cite{trefethen2000spectral,shen2011spectral}.
\paragraph{Spectral truncation.}

\begin{lemma}[Spectral projection]\label{lem:ts-spectral}
For the Fourier box $\mathcal K_K=\{\boldsymbol k\in\mathbb Z^d:|k_j|\le K\}$ and $u\in H^s(\mathbb T^d)$, $s>m\ge0$,
\begin{equation}\label{eq:ts-fourier}
  \|u-P_Ku\|_{H^m}\le (1+K^2)^{(m-s)/2}\|u\|_{H^s}.
\end{equation}
If $|\widehat u_{\boldsymbol k}| \le B e^{-a|\boldsymbol k|_1}$, then $\|u-P_Ku\|_{L^2}\le C(d,a)B e^{-aK}$. For the product Jacobi weight $w$, denote the base and shifted weighted norms by $\|v\|_w^2=\int_\Omega |v|^2w\,d\boldsymbol x$ and $\|v\|_{j,s}^2=\int_\Omega |v|^2(1-x_j^2)^sw\,d\boldsymbol x$, respectively. For tensor Jacobi projection of degree at most $p_j$ in axis $j$, integer $s_j\ge1$ with $p_j\ge s_j-1$, and the corresponding Jacobi-weighted derivative space,
\begin{equation}\label{eq:ts-jacobi}
  \|u-P_{\boldsymbol p}u\|_w \le \sum_{j=1}^d C_{\alpha_j,\beta_j,s_j}(p_j+1)^{-s_j} \|\partial_j^{s_j}u\|_{j,s_j}.
\end{equation}
The factor $(1-x_j^2)^{s_j}$ shifts the $j$th Jacobi parameters from $(\alpha_j,\beta_j)$ to $(\alpha_j+s_j,\beta_j+s_j)$.
\end{lemma}

\hypertarget{proof:ts-spectral-fourier}{}
\begin{proof}[Proof of the Fourier part of Lemma~\ref{lem:ts-spectral}]
Define
\begin{equation}
  \|u\|_{H^s}^2 =\sum_{\boldsymbol k\in\mathbb Z^d} (1+|\boldsymbol k|_2^2)^s|\widehat u_{\boldsymbol k}|^2.
\end{equation}
For $\boldsymbol k\notin\mathcal K_K$, at least one $|k_j|\ge K+1$, hence $1+|\boldsymbol k|_2^2\ge1+K^2$.  Because $m-s<0$,
\begin{align*}
\|u-P_Ku\|_{H^m}^2
 &=\sum_{\boldsymbol k\notin\mathcal K_K}
 (1+|\boldsymbol k|_2^2)^m|\widehat u_{\boldsymbol k}|^2\\
 &\le(1+K^2)^{m-s}\sum_{\boldsymbol k\notin\mathcal K_K}
 (1+|\boldsymbol k|_2^2)^s|\widehat u_{\boldsymbol k}|^2
 \le(1+K^2)^{m-s}\|u\|_{H^s}^2.
\end{align*}
Taking square roots proves Equation~\eqref{eq:ts-fourier}.  Under the analytic assumption, applying the union bound over the $d$ events $|k_j|>K$ in the coefficient sum gives
\begin{align*}
\|u-P_Ku\|_{L^2}^2
\le B^2\sum_{\boldsymbol k\notin\mathcal K_K}
e^{-2a|\boldsymbol k|_1}\le dB^2\left(2\sum_{n=K+1}^\infty e^{-2an}\right)
 \left(\sum_{n\in\mathbb Z}e^{-2a|n|}\right)^{d-1}.
\end{align*}
The geometric series yields the stated $C(d,a)B e^{-aK}$ bound.
\end{proof}

\hypertarget{proof:ts-spectral-jacobi}{}
\begin{proof}[Proof of the Jacobi part of Lemma~\ref{lem:ts-spectral}]
Let $\phi_n^{(\alpha,\beta)}=J_n^{(\alpha,\beta)}/\sqrt{h_n^{(\alpha,\beta)}}$ be orthonormal Jacobi polynomials, and denote the weighted norm by $\|v\|_{\alpha,\beta}^2=\int_{-1}^1 |v(x)|^2(1-x)^{\alpha}(1+x)^{\beta}\,dx$.  For integer $s\ge1$ the Jacobi derivative identity and orthogonality at shifted parameters give
\begin{equation}\label{eq:app-jacobi-derivative}
  \left\| \partial_x^s\sum_{n\ge0}c_n\phi_n^{(\alpha,\beta)} \right\|_{\alpha+s,\beta+s}^2 =\sum_{n\ge s}\Lambda_{n,s}^{(\alpha,\beta)}|c_n|^2.
\end{equation}
Here
\begin{equation}
  \Lambda_{n,s}^{(\alpha,\beta)} =\prod_{\ell=0}^{s-1}(n-\ell)(n+\alpha+\beta+1+\ell).
\end{equation}
For $\alpha,\beta>-1$ and $n>p\ge s-1$, every factor in this product is positive and $\Lambda_{n,s}^{(\alpha,\beta)}\ge c_{\alpha,\beta,s}(p+1)^{2s}$, with a positive constant independent of $p$.  Parseval therefore gives
\begin{align*}
\|v-P_pv\|_{\alpha,\beta}^2
 &=\sum_{n>p}|c_n|^2\\
 &\le c_{\alpha,\beta,s}^{-1}(p+1)^{-2s}
 \|\partial_x^sv\|_{\alpha+s,\beta+s}^2.
\end{align*}
Tensor projections commute and each is an $L^2_w$ contraction. Their telescoping identity is
\begin{equation}
  I-\prod_{j=1}^dP_{p_j}^{(j)} =\sum_{j=1}^d \left(\prod_{\ell<j}P_{p_\ell}^{(\ell)}\right) \left(I-P_{p_j}^{(j)}\right).
\end{equation}
Apply the triangle inequality, contraction in the other axes, and the one-dimensional estimate to each summand. Fubini gives
Equation~\eqref{eq:ts-jacobi}.
\end{proof}

\paragraph{Rank approximation.}
After spectral projection, increasing the admissible CP rank enlarges the coefficient approximation space. This gives monotonicity, but no rate without additional structure in the target tensor.

\begin{lemma}[Rank error]\label{lem:ts-rank}
$E_{\mathrm{rank}}(N,r+1)\le E_{\mathrm{rank}}(N,r)$, and the error is zero when $r\ge\operatorname{rank}_{\rm CP}(\widehat{\mathcal U}_N)$. No algebraic or exponential rate in $r$ follows without a solution-specific compressibility assumption.
\end{lemma}
\hypertarget{proof:ts-rank}{}
\begin{proof}[Proof of Lemma~\ref{lem:ts-rank}]
The feasible tensors of rank at most $r$ are contained in those of rank at most $r+1$, so the infimum cannot increase.  If $r$ reaches the actual CP rank of $\widehat{\mathcal U}_N$, the target tensor is feasible and has zero error.  Nonnegativity then forces the infimum to be zero.
\end{proof}

\paragraph{Base-model error decomposition.}
Theorem~\ref{thm:ts-decomp} combines the two approximation errors with the distance from the trained tensor to a near-best rank-$r$ tensor. In the notation above, its bound is $\|u-u_\theta\|_H\le E_{\mathrm{spec}}(N)+E_{\mathrm{rank}}(N,r)+\eta+E_{\mathrm{opt}}(\theta;\widehat{\mathcal V}_\eta)$. The isometry makes this coefficient-space training gap comparable to solution error; it does not yet connect that gap to the PDE loss.

\hypertarget{proof:ts-decomp}{}
\begin{proof}[Proof of Theorem~\ref{thm:ts-decomp}]
Using $u_N=P_Nu$ and the synthesis isometry Equation~\eqref{eq:app-isometry}, we obtain
\begin{equation}
  \begin{aligned}
    \|u-u_\theta\|_H &\le\|u-u_N\|_H+ \|u_N-S_N\widehat{\mathcal U}_\theta\|_H\\
    &=E_{\rm spec}(N)+ \|\widehat{\mathcal U}_N-\widehat{\mathcal U}_\theta\|_F\\
    &\le E_{\rm spec}(N)+ \|\widehat{\mathcal U}_N-\widehat{\mathcal V}_\eta\|_F+ \|\widehat{\mathcal V}_\eta-\widehat{\mathcal U}_\theta\|_F\\
    &\le E_{\rm spec}(N)+E_{\rm rank}(N,r)+\eta +E_{\rm opt}(\theta;\widehat{\mathcal V}_\eta).
  \end{aligned}
\end{equation}
The second inequality applies the triangle inequality in coefficient space.
\end{proof}

\subsection{Residual Stability and Spectral Weighting}
The approximation decomposition describes what the representation can resolve. To assess a trained solution from its PDE residual, we additionally need stability of the differential operator and control of the modes measured by the objective.
\paragraph{Full residual and weighted loss.}
For Theorem~\ref{thm:ts-stability}, let $X_0$ denote the admissible errors satisfying the homogeneous essential constraints. Assume $Au=f$, linearity of $A$, and $\|e\|_X\le C_{\mathrm{stab}}\|Ae\|_Y$ for $e\in X_0$. If $u_\theta$ satisfies the same essential constraints as $u$, the theorem gives $\|u-u_\theta\|_X\le C_{\mathrm{stab}}\|Au_\theta-f\|_Y$. We next express this full-residual estimate in terms of the weighted coefficient objective.

For a stationary residual, write $R=Au_\theta-f$ and let $\mathcal R_{\boldsymbol k}$ be its coefficients in a finite orthonormal residual basis. For weights $0<w_{\min}\le w_{\boldsymbol k}\le w_{\max}$ and a \emph{complete} mode set $\mathcal K$ of size $n$, the full-mode version of the preconditioned loss in Section~\ref{sec:ts-method}, $\mathcal L_{\mathrm{PDE}}^{\mathrm{prec,full}}=n^{-1}\sum_{\boldsymbol k\in\mathcal K} |w_{\boldsymbol k}\mathcal R_{\boldsymbol k}|^2$, implies
\begin{equation}\label{eq:ts-loss}
  \|u-u_\theta\|_X\le C_{\rm stab} \bigl(w_{\min}^{-1}\sqrt{n\mathcal L_{\mathrm{PDE}}^{\mathrm{prec,full}}} +\|R_{\mathcal K^c}\|_Y\bigr),
\end{equation}
provided the coefficient split is orthogonal in $Y$. For an empirical subset $\mathcal S\subsetneq\mathcal K$, the analogous loss $\mathcal L_{\mathcal S} =|\mathcal S|^{-1}\sum_{\boldsymbol k\in\mathcal S}|w_{\boldsymbol k}\mathcal R_{\boldsymbol k}|^2$ controls only $R_{\mathcal S}$; an independent bound on $R_{\mathcal S^c}$ is required. A space--time sampled average likewise requires quadrature, generalization, or a continuous residual estimate. Boundary and initial condition mismatches must be included unless imposed exactly.

\hypertarget{proof:ts-stability}{}
\begin{proof}[Proof of Theorem~\ref{thm:ts-stability} and Equation~\eqref{eq:ts-loss}]
For $e=u_\theta-u\in X_0$, linearity and $Au=f$ give
\begin{equation}
  Ae=Au_\theta-f=R.
\end{equation}
The assumed stability inequality yields Equation~\eqref{eq:ts-full-res}.  For the retained coefficient vector,
\begin{equation}
  |\mathcal R_{\boldsymbol k}| \le w_{\min}^{-1} |w_{\boldsymbol k}\mathcal R_{\boldsymbol k}|.
\end{equation}
Consequently,
\begin{equation}
  \|R_{\mathcal K}\|_Y \le w_{\min}^{-1} \left(\sum_{\boldsymbol k\in\mathcal K} |w_{\boldsymbol k}\mathcal R_{\boldsymbol k}|^2\right)^{1/2}=w_{\min}^{-1} \sqrt{n\mathcal L_{\mathrm{PDE}}^{\mathrm{prec,full}}}.
\end{equation}
Orthogonality and the triangle inequality give
\begin{equation}
  \|R\|_Y \le\|R_{\mathcal K}\|_Y+\|R_{\mathcal K^c}\|_Y.
\end{equation}
Substitution proves the loss bound. If training uses $\mathcal S\subsetneq\mathcal K$, the same argument applies only to $R_{\mathcal S}$. Indeed, a nonzero residual supported on one omitted mode has zero empirical loss and a positive full norm. For a sampling distribution $p_{\boldsymbol k}>0$, an importance-weighted expectation recovers the full sum as
\begin{equation}
  \mathbb E_p \left[ \frac{|w_{\boldsymbol k}\mathcal R_{\boldsymbol k}|^2} {p_{\boldsymbol k}} \right] =\sum_{\boldsymbol k\in\mathcal K} |w_{\boldsymbol k}\mathcal R_{\boldsymbol k}|^2.
\end{equation}
An uncorrected expectation gives only $\sum_{\boldsymbol k\in\mathcal K} p_{\boldsymbol k}|w_{\boldsymbol k}\mathcal R_{\boldsymbol k}|^2$ and needs $p_{\min}>0$ to bound the sum.
\end{proof}

\paragraph{Projected residuals.}
A projected residual instead compares the approximation with a discrete Galerkin solution. Its consistency error must be retained separately, as the following bound shows.

\begin{proposition}[A projected-residual alternative]\label{prop:ts-discrete}
Suppose $A:\operatorname{dom}(A)\to H$, $V_N\subset\operatorname{dom}(A)$, $f\in H$, and $A_N=P_NA|_{V_N}$ is injective, with $\|z\|_H\le\beta_N^{-1}\|A_Nz\|_H$ for $z\in V_N$. Let $u_N^G\in V_N$ solve $A_Nu_N^G=P_Nf$.  For $v\in V_N$,
\begin{equation}\label{eq:ts-discrete}
  \|u-v\|_H\le\|u-u_N^G\|_H+ \beta_N^{-1}\|P_N(Av-f)\|_H.
\end{equation}
Thus a complete weighted Galerkin residual controls error relative to the discrete solution, with a separate consistency error and a possibly $N,d$-dependent $\beta_N$.
\end{proposition}
\hypertarget{proof:ts-discrete}{}
\begin{proof}[Proof of Proposition~\ref{prop:ts-discrete}]
Set $z=v-u_N^G\in V_N$.  Since $A_Nu_N^G=P_Nf$, $A_Nz=P_N(Av-f)$.  Discrete stability gives
\begin{equation}
  \|z\|_H \le\beta_N^{-1}\|P_N(Av-f)\|_H.
\end{equation}
The triangle inequality
\begin{equation}
  \|u-v\|_H \le\|u-u_N^G\|_H+\|u_N^G-v\|_H
\end{equation}
then proves Equation~\eqref{eq:ts-discrete}.  A full weighted average bounds the projected residual using the weight lower bound as in Theorem~\ref{thm:ts-stability}.
\end{proof}

\paragraph{Periodic Poisson and ideal weights.}
For zero-mean periodic Poisson, Fourier diagonalization identifies an exact relation between the full residual and the solution error. This gives a concrete case in which operator-aware weights recover the solution norm.

\begin{corollary}[Zero-mean periodic Poisson]\label{cor:ts-poisson}
Use normalized Haar measure on $\mathbb T^d$, set the zero mode of $u$ and $u_\theta$ to zero, and let $R=\Delta u_\theta-f$ with $\Delta u=f$. Write $\mathcal R_{\boldsymbol k}=\widehat R_{\boldsymbol k}$ and $e=u_\theta-u$.  Then, for $\boldsymbol k\ne\boldsymbol0$,
\begin{equation}\label{eq:ts-poisson}
  \begin{aligned}
    \widehat e_{\boldsymbol k} &= -\frac{\mathcal R_{\boldsymbol k}}{|\boldsymbol k|_2^2},\\
    \|e\|_{L^2}^2 &= \sum_{\boldsymbol k\ne\boldsymbol0} \frac{|\mathcal R_{\boldsymbol k}|^2}{|\boldsymbol k|_2^4},\\
    \|\nabla e\|_{L^2}^2 &= \sum_{\boldsymbol k\ne\boldsymbol0} \frac{|\mathcal R_{\boldsymbol k}|^2}{|\boldsymbol k|_2^2}.
  \end{aligned}
\end{equation}
Consequently, ideal weights $w_{\boldsymbol k}=|\boldsymbol k|_2^{-2}$ make the \emph{complete} weighted residual norm equal to the coefficient $L^2$ error. For Jacobi bases, the coefficient-space operator scale provides a preconditioner.
\end{corollary}
\hypertarget{proof:ts-poisson}{}
\begin{proof}[Proof of Corollary~\ref{cor:ts-poisson}]
Since $\Delta e=R$ and both solutions have zero mean, Fourier diagonalization gives
\begin{equation}
  -|\boldsymbol k|_2^2\widehat e_{\boldsymbol k} =\mathcal R_{\boldsymbol k},
  \quad \boldsymbol k\ne\boldsymbol0.
\end{equation}
Parseval on normalized $\mathbb T^d$ then gives
\begin{equation}
  \begin{aligned}
    \|e\|_{L^2}^2 &=\sum_{\boldsymbol k\ne\boldsymbol0} \frac{|\mathcal R_{\boldsymbol k}|^2}{|\boldsymbol k|_2^4},\\
    \|\nabla e\|_{L^2}^2 &=\sum_{\boldsymbol k\ne\boldsymbol0} |\boldsymbol k|_2^2|\widehat e_{\boldsymbol k}|^2 =\sum_{\boldsymbol k\ne\boldsymbol0} \frac{|\mathcal R_{\boldsymbol k}|^2}{|\boldsymbol k|_2^2}.
  \end{aligned}
\end{equation}
The latter is the squared homogeneous $H^{-1}$ norm of $R$. The ideal weight $|\boldsymbol k|_2^{-2}$ makes each weighted residual coefficient equal to $-\widehat e_{\boldsymbol k}$, so its full $\ell^2$ norm is exactly the solution $L^2$ error.  A generic positive weight only gives
\begin{equation}
  \min_{\boldsymbol k\in\mathcal K} |w_{\boldsymbol k}||\boldsymbol k|_2^2\,\|e\|_{L^2} \le\|WR\|_2 \le \max_{\boldsymbol k\in\mathcal K} |w_{\boldsymbol k}||\boldsymbol k|_2^2\,\|e\|_{L^2}
\end{equation}
on a fixed finite retained set, provided all relevant coefficients are included.
\end{proof}

\paragraph{Evolution equations and initial mismatches.}
For time-dependent problems, a residual bound must also account for initial data. The heat semigroup gives an $L^2$ estimate involving one initial mismatch; the wave propagator requires both the initial value and velocity.

\begin{proposition}[Heat equation with initial mismatch]\label{prop:app-heat}
Let $u_t-\Delta u=0$ and $v_t-\Delta v=R$ on $\mathbb T^d\times(0,T)$, with periodic spatial boundary conditions, $R\in L^2(0,T;L^2)$, and
$e_0=v(0)-u(0)\in L^2$. Then
\begin{equation}\label{eq:app-heat}
  \sup_{0\le t\le T}\|v(t)-u(t)\|_{L^2} \le\|e_0\|_{L^2}+\sqrt T\,\|R\|_{L^2(0,T;L^2)}.
\end{equation}
\end{proposition}
\begin{proof}
The heat semigroup is an $L^2$ contraction because its Fourier multiplier is $e^{-|\boldsymbol k|_2^2t}$ with modulus at most one. Duhamel's formula for $e=v-u$ is
\begin{equation}
  e(t)=e^{t\Delta}e_0+ \int_0^t e^{(t-s)\Delta}R(s)\,ds.
\end{equation}
Taking norms, applying the contraction, and then Cauchy--Schwarz gives
\begin{equation}
    \|e(t)\|_{L^2} \le\|e_0\|_{L^2}+\int_0^t\|R(s)\|_{L^2}\,ds \le\|e_0\|_{L^2} +\sqrt t\,\|R\|_{L^2(0,t;L^2)}.
\end{equation}
The estimate follows upon taking the supremum.
\end{proof}

\begin{proposition}[Wave equation with both initial mismatches]\label{prop:app-wave}
Let $u_{tt}-\Delta u=0$ and $v_{tt}-\Delta v=R$ on the periodic domain, with $e_0=v(0)-u(0)$, $e_1=v_t(0)-u_t(0)$, and
$R\in L^2(0,T;L^2)$. Then
\begin{equation}\label{eq:app-wave}
  \|v(t)-u(t)\|_{L^2} \le\|e_0\|_{L^2}+t\|e_1\|_{L^2} +\frac{t^{3/2}}{\sqrt3}\|R\|_{L^2(0,t;L^2)}.
\end{equation}
The energy also satisfies
$E(t)^{1/2}\le E(0)^{1/2}+
2^{-1/2}\int_0^t\|R(s)\|_{L^2}ds$,
$E=\frac12(\|e_t\|_{L^2}^2+\|\nabla e\|_{L^2}^2)$.
\end{proposition}
\begin{proof}
For $L=-\Delta\ge0$, the modal formula is
\begin{equation}
  \begin{aligned}
    e(t)={}&\cos(t\sqrt L)e_0 +\frac{\sin(t\sqrt L)}{\sqrt L}e_1 + \int_0^t \frac{\sin((t-s)\sqrt L)}{\sqrt L}R(s)\,ds,
  \end{aligned}
\end{equation}
where the quotient equals $t$ at the zero mode.  The spectral bounds
$|\cos(t\sqrt\lambda)|\le1$ and
$|\sin(t\sqrt\lambda)|/\sqrt\lambda\le t$ for $\lambda\ge0$
give
\begin{equation}
  \|e(t)\|_{L^2} \le\|e_0\|_{L^2}+t\|e_1\|_{L^2} +\int_0^t(t-s)\|R(s)\|_{L^2}\,ds.
\end{equation}
Cauchy--Schwarz and
$\int_0^t(t-s)^2ds=t^3/3$ prove Equation~\eqref{eq:app-wave}.
For energy, test $e_{tt}-\Delta e=R$ against $e_t$ to obtain
\begin{equation}
  E'(t)=\langle R,e_t\rangle \le\sqrt{2E(t)}\,\|R(t)\|_{L^2}.
\end{equation}
Where $E(t)>0$, divide by $2\sqrt E$ and integrate; the zero-energy case follows by regularization.
\end{proof}

\paragraph{Conditional nonlinear stability.}
The linear stability theorem does not automatically extend to nonlinear PDEs. The next two estimates specify sufficient boundedness or coercivity conditions for the reaction and Allen--Cahn operators.

\begin{proposition}[Conditional nonlinear stability]\label{prop:app-nonlinear}
(i) For $F(z)=\rho z(1-z)$, $\rho>0$, suppose the exact and approximate reaction solutions are bounded in absolute value by $B$ on $[0,T]$ and
$R=v_t-F(v)\in L^1(0,T;L^2)$. Then
\begin{equation}\label{eq:app-reaction}
  \|v(t)-u(t)\|_{L^2}\le e^{\rho(1+2B)t} \left(\|v(0)-u(0)\|_{L^2} +\int_0^t\|R(s)\|_{L^2}ds\right).
\end{equation}
(ii) On $[-1,1]^d$ with homogeneous Dirichlet conditions, let $A(z)=-\varepsilon^2\Delta z+z^3-z$, and let $\lambda_1$ be the first Dirichlet eigenvalue of $-\Delta$. If $\varepsilon^2>\lambda_1^{-1}$, then two solutions with full residual $R=A(v)-A(u)$ satisfy
\begin{equation}\label{eq:app-allen}
  \|v-u\|_{H_0^1,\nabla} \le(\varepsilon^2-\lambda_1^{-1})^{-1} \|R\|_{H^{-1}},
\end{equation}
where $\|z\|_{H_0^1,\nabla}=\|\nabla z\|_{L^2}$.
\end{proposition}
\begin{proof}
(i) On $[-B,B]$, $|F'(z)|=\rho|1-2z|\le\rho(1+2B)=L$. For $e=v-u$, $e_t=F(v)-F(u)+R$.  The upper Dini derivative of
$\|e\|_{L^2}$ satisfies
\begin{equation}
  D^+\|e\|_{L^2} \le L\|e\|_{L^2}+\|R\|_{L^2}.
\end{equation}
Multiplying by $e^{-Lt}$, integrating, and dropping a factor $e^{-Ls}\le1$ proves Equation~\eqref{eq:app-reaction}.
(ii) Put $e=v-u$.  Testing the residual difference with $e$ yields
\begin{align*}
\langle R,e\rangle
 &=\varepsilon^2\|\nabla e\|_{L^2}^2
 +\int(v^3-u^3)(v-u)\,dx-\|e\|_{L^2}^2\\
 &\ge(\varepsilon^2-\lambda_1^{-1})\|\nabla e\|_{L^2}^2.
\end{align*}
The cubic integrand is nonnegative because
\begin{equation}
  (v^3-u^3)(v-u) =(v-u)^2(v^2+uv+u^2)\ge0.
\end{equation}
Poincare's inequality supplied the $\lambda_1^{-1}$ term. Duality bounds the left side by $\|R\|_{H^{-1}}\|\nabla e\|_{L^2}$; division proves Equation~\eqref{eq:app-allen}.
\end{proof}

\subsection{Low-Rank Correction and De-aliasing}
We now combine the approximation and residual viewpoints for staged corrections. Increasing rank enlarges capacity, while reducing the training objective controls solution error only through an appropriate stability estimate.
\paragraph{Accumulated rank.}

Let $\widehat{\mathcal U}^{[0]}$ have CP rank at most $r_0$, and let
$\widehat{\mathcal U}^{[m]} = \widehat{\mathcal U}^{[m-1]}+\Delta\widehat{\mathcal U}^{(m)}$ with
$\operatorname{rank}_{\rm CP}(\Delta\widehat{\mathcal U}^{(m)})\le r_m$.

\begin{lemma}[Rank accumulation]\label{lem:ts-rank-sum}
$\operatorname{rank}_{\rm CP}(\widehat{\mathcal U}^{[M]}) \le r_0+\sum_{m=1}^M r_m$. Earlier factors remain fixed during subsequent correction stages.
\end{lemma}
\hypertarget{proof:ts-rank-sum}{}
\begin{proof}[Proof of Lemma~\ref{lem:ts-rank-sum}]
Concatenate the $r_0$ rank-one terms in $\widehat{\mathcal U}^{[0]}$ with the at most $r_m$ terms from each $\Delta\widehat{\mathcal U}^{(m)}$.  This gives an explicit CP expansion with at most $r_0+\sum_mr_m$ terms, so the minimum CP rank cannot exceed that number.
\end{proof}

\paragraph{What objective reduction guarantees.}
Because a zero correction leaves the current iterate unchanged, near-minimization over an admissible correction class controls the stage objective. Transferring that control to solution error requires norm equivalence; in general, it does not yield a contraction.

\begin{proposition}[Residual Reduction and Solution Error]
\label{prop:app-stage-objective}
Fix a complete coefficient residual objective
\begin{equation}
  J(v)=\|W\widehat{P_N(Av-f)}\|_2,
\end{equation}
where $W$ applies the spectral weights to the coefficient vector. Let $\mathcal C_m$ be an admissible correction class containing zero. If stage $m$ returns $c_m\in\mathcal C_m$ with
\begin{equation}
  J(v_{m-1}+c_m) \le\inf_{c\in\mathcal C_m}J(v_{m-1}+c)+\zeta_m,
\end{equation}
then $J(v_m)\le J(v_{m-1})+\zeta_m$. If, on the relevant finite-dimensional error space,
\begin{equation}
  a\|e\|_H\le J(u+e)\le b\|e\|_H
  \quad (a>0),
\end{equation}
then
\begin{equation}
  \|e_m\|_H \le\frac{b}{a}\|e_{m-1}\|_H+\frac{\zeta_m}{a}.
\end{equation}
\end{proposition}
\begin{proof}
Zero is an admissible correction, so the infimum is at most $J(v_{m-1})$.  This gives the objective inequality.  The norm-equivalence inequalities then yield
\begin{equation}
  \begin{aligned}
    a\|e_m\|_H &\le J(v_m)\\
    &\le J(v_{m-1})+\zeta_m\\
    &\le b\|e_{m-1}\|_H+\zeta_m.
  \end{aligned}
\end{equation}
The factor $b/a$ quantifies how residual reduction transfers to the solution-error bound. For $a=b$, an exact stage minimization yields non-increasing solution error. For the zero-mean Poisson problem with complete modes and ideal weights $w_{\boldsymbol k}=|\boldsymbol k|_2^{-2}$, one has $a=b=1$. Comparisons across sampled or changing training objectives use a common fixed validation norm.
\end{proof}

\paragraph{Conditional coefficient-error reduction.}
A stronger recurrence follows if each stage can approximate the remaining coefficient error to a prescribed tolerance. The following theorem makes this extra compressibility assumption explicit, rather than deriving it from a decrease in sampled training loss.

\begin{theorem}[Conditional correction recurrence]\label{thm:ts-correction}
Fix a target coefficient tensor $\widehat{\mathcal U}_N^*$, put $\widehat{\mathcal E}^{[m]} =\widehat{\mathcal U}_N^*-\widehat{\mathcal U}^{[m]}$, and define $\delta_m =\inf_{\operatorname{rank}_{\rm CP}(\widehat{\mathcal V})\le r_m} \|\widehat{\mathcal E}^{[m-1]}-\widehat{\mathcal V}\|_F$. If stage $m$ returns a correction satisfying $\|\widehat{\mathcal E}^{[m]}\|_F\le\delta_m+\varepsilon_m$, then $\|\widehat{\mathcal E}^{[m]}\|_F \le\|\widehat{\mathcal E}^{[m-1]}\|_F+\varepsilon_m$. An attained exact best correction gives non-increase.  If, additionally, $\delta_m\le q_m\|\widehat{\mathcal E}^{[m-1]}\|_F$ with $0\le q_m<1$, then
\begin{equation}\label{eq:ts-recurrence}
  \|\widehat{\mathcal E}^{[M]}\|_F\le \left(\prod_{j=1}^Mq_j\right)\|\widehat{\mathcal E}^{[0]}\|_F +\sum_{m=1}^M\varepsilon_m\prod_{j=m+1}^Mq_j.
\end{equation}
For $q_m\le q<1$ and $\varepsilon_m\le\varepsilon$, this becomes $\|\widehat{\mathcal E}^{[M]}\|_F \le q^M\|\widehat{\mathcal E}^{[0]}\|_F+ \varepsilon(1-q^M)/(1-q)$. The recurrence assumes compressibility of the coefficient error and a stagewise approximation tolerance. For PDE-residual training, the stagewise coefficient-error condition requires a residual-to-error stability estimate. Proposition~\ref{prop:app-stage-objective} above states what a decrease in the training objective alone guarantees.
\end{theorem}

\hypertarget{proof:ts-correction}{}
\begin{proof}[Proof of Theorem~\ref{thm:ts-correction}]
By the definition of the cumulative update,
\begin{equation}
  \begin{aligned}
    \widehat{\mathcal E}^{[m]} &=\widehat{\mathcal U}_N^*- \bigl(\widehat{\mathcal U}^{[m-1]} +\Delta\widehat{\mathcal U}^{(m)}\bigr)\\
    &=\widehat{\mathcal E}^{[m-1]} -\Delta\widehat{\mathcal U}^{(m)}.
  \end{aligned}
\end{equation}
An attained exact minimizer therefore has $\|\widehat{\mathcal E}^{[m]}\|_F=\delta_m$.  Because the zero tensor has rank zero and is admissible, $\delta_m\le\|\widehat{\mathcal E}^{[m-1]}\|_F$. With inexactness,
\begin{equation}
  \|\widehat{\mathcal E}^{[m]}\|_F \le\delta_m+\varepsilon_m \le\|\widehat{\mathcal E}^{[m-1]}\|_F+\varepsilon_m.
\end{equation}
Under the extra compressibility inequality, $a_m:=\|\widehat{\mathcal E}^{[m]}\|_F$ satisfies $a_m\le q_ma_{m-1}+\varepsilon_m$. Induction yields
\begin{equation}
  a_M\le \left(\prod_{j=1}^M q_j\right)a_0 +\sum_{m=1}^M\varepsilon_m \prod_{j=m+1}^M q_j.
\end{equation}
The formula is true for $M=1$; multiplying the $M-1$ formula by $q_M$ and adding $\varepsilon_M$ gives the $M$ formula.  If $q_m\le q<1$ and $\varepsilon_m\le\varepsilon$, bound each product by the corresponding power of $q$ and sum the geometric series.
\end{proof}

\paragraph{Total approximation error and certification.}
Taking the projected exact solution as the correction target combines the recurrence with spectral truncation. The residual estimate provides a separate certificate for the same final iterate whenever its stability assumptions hold.

\begin{theorem}[Combined approximation and certification]\label{thm:ts-combined}
Take $\widehat{\mathcal U}_N^*=\widehat{\mathcal U}_N$ in Theorem~\ref{thm:ts-correction}, assume its stage hypotheses, and set $u_{\rm TS}^{[M]}=S_N\widehat{\mathcal U}^{[M]}$ for the synthesis map in Equation~\eqref{eq:app-isometry}.  Then
\begin{equation}\label{eq:ts-combined}
  \begin{aligned}
    \|u-u_{\rm TS}^{[M]}\|_H &\le E_{\mathrm{spec}}(N)
    +\left(\prod_{j=1}^Mq_j\right) \|\widehat{\mathcal U}_N-\widehat{\mathcal U}^{[0]}\|_F \\
    &\quad +\sum_{m=1}^M\varepsilon_m\prod_{j=m+1}^Mq_j.
  \end{aligned}
\end{equation}
If $\widehat{\mathcal U}^{[0]}=\widehat{\mathcal U}_\theta$, the base mismatch can be bounded by $E_{\mathrm{rank}}(N,r_0)+\eta+ E_{\mathrm{opt}}(\theta;\widehat{\mathcal V}_\eta)$ as in Theorem~\ref{thm:ts-decomp}.  Separately, whenever the full-residual assumptions of Theorem~\ref{thm:ts-stability} hold, its a posteriori bound applies to the same $u_{\rm TS}^{[M]}$; the smaller of the two valid bounds may be used.
\end{theorem}
\hypertarget{proof:ts-combined}{}
\begin{proof}[Proof of Theorem~\ref{thm:ts-combined}]
The triangle inequality and Equation~\eqref{eq:app-isometry} give
\begin{equation}
  \begin{aligned}
    \|u-u_{\rm TS}^{[M]}\|_H &\le\|u-P_Nu\|_H+ \|\widehat{\mathcal U}_N-\widehat{\mathcal U}^{[M]}\|_F\\
    &=E_{\rm spec}(N)+\|\widehat{\mathcal E}^{[M]}\|_F.
  \end{aligned}
\end{equation}
Insert Equation~\eqref{eq:ts-recurrence} with $\widehat{\mathcal U}_N^*=\widehat{\mathcal U}_N$ to get Equation~\eqref{eq:ts-combined}. The same two triangle inequalities as in Theorem~\ref{thm:ts-decomp} bound $\|\widehat{\mathcal E}^{[0]}\|_F$ by $E_{\rm rank}(N,r_0)+\eta+ E_{\rm opt}(\theta;\widehat{\mathcal V}_\eta)$, with $\widehat{\mathcal U}^{[0]}=\widehat{\mathcal U}_\theta$. The full-residual estimate independently bounds the same left-hand side through PDE stability and may be compared with the approximation bound.
\end{proof}

\paragraph{De-aliasing and exact retained coefficients.}
The nonlinear cost analysis assumes that factorwise products are projected without aliasing into the retained modes. We conclude by stating and proving sufficient Fourier padding and Jacobi quadrature conditions.

\begin{proposition}[Exact retained nonlinear coefficients]
\label{prop:app-dealias}
(i) Let every Fourier factor be supported in $[-K,K]$, with $N=2K+1$. For a polynomial nonlinearity of degree $s$, an $M$-point periodic FFT product has the correct coefficients at all retained modes $[-K,K]$ if $M>(s+1)K$. In particular, a quadratic product is exact on the retained set if $M>3K$; $M=\lceil3N/2\rceil$ suffices.
(ii) If each Jacobi factor has polynomial degree at most $p$, the exact projection of its degree-$s$ product onto degrees at most $p$ is obtained by $Q$-point Gauss--Jacobi quadrature whenever
\begin{equation}
  2Q-1\ge(s+1)p,
  \quad Q\ge\left\lceil\frac{(s+1)p+1}{2}\right\rceil.
\end{equation}
For a cubic product the bound is $Q\ge2p+1$, and any sequential multiplication must keep enough intermediate degree to avoid truncation before the final projection.
\end{proposition}
\begin{proof}
(i) The exact $s$-fold Fourier product has support $[-sK,sK]$. The $M$-point discrete Fourier transform returns at retained index $n\in[-K,K]$ the sum of exact coefficients at $n+\ell M$ over integers $\ell$.  If $M>(s+1)K$, then
\begin{equation}
  \begin{aligned}
    n+M&>-K+(s+1)K=sK,\\
    n-M&<K-(s+1)K=-sK.
  \end{aligned}
\end{equation}
Every nonzero alias is outside the product support, so only the true coefficient at $n$ remains.  For $s=2$ and odd $N=2K+1$, $\lceil3N/2\rceil>3K$. Even-length FFTs require an explicit Nyquist convention.
(ii) A degree-$s$ polynomial in an input of degree $p$ has degree at most $sp$.  Multiplying it by a test Jacobi polynomial of degree at most $p$ gives degree at most $(s+1)p$.  Gauss--Jacobi quadrature integrates the Jacobi weight times every polynomial of degree at most $2Q-1$ exactly. The stated inequality therefore makes each projected coefficient exact. If an intermediate product is projected to degree $p$ before forming a cubic, degrees $p+1,\ldots,2p$ can influence its later low-degree coefficients, so quadrature exactness alone cannot repair that early loss.
\end{proof}

\paragraph{Dependence on dimension.}
The complexity bounds are stated for a prescribed CP rank. The rank required for a target accuracy and the associated norm constants may depend on dimension. Thus the factorwise storage and operation counts do not by themselves imply dimension-independent approximation accuracy.

\section{Computational Cost} \label{sec:comp cost}
Figure~\ref{fig:computational_cost} compares the computational cost per training update on the 30-dimensional wave equation. TensorSpectra requires only $0.529$ GFLOPs for the base stage and $0.468$ GFLOPs for the correction stage, which is substantially lower than PINN ($47.34$ GFLOPs), SPINN ($33.16$ GFLOPs), and PirateNet ($242.22$ GFLOPs), and is also slightly lower than SINN ($0.716$ GFLOPs). These results indicate that the low-rank spectral formulation provides a favorable balance between computational cost and solution accuracy in high-dimensional PDE solving.
\begin{figure}[h]
    \centering
    \includegraphics[width=0.8\textwidth]{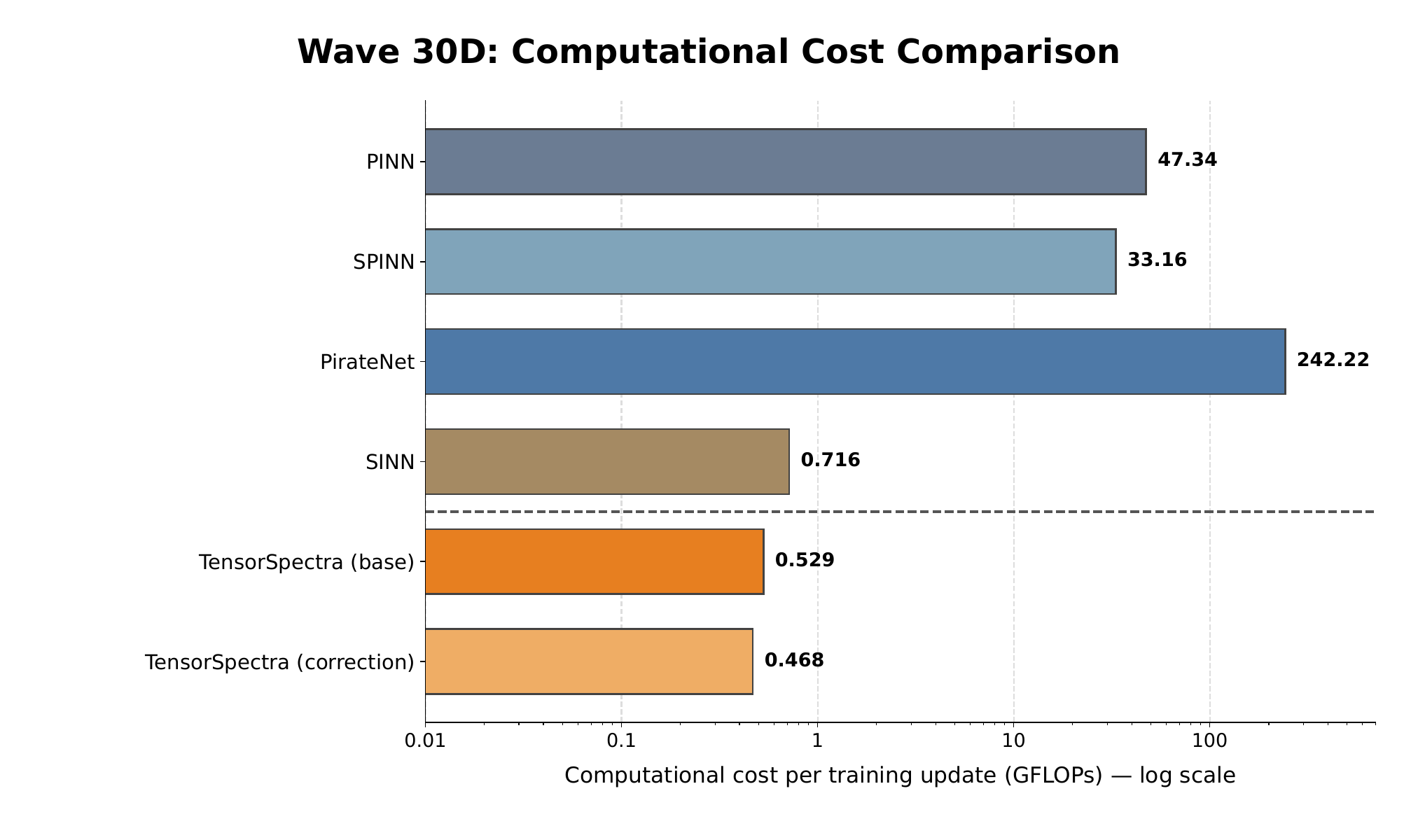}
    \caption{Computational cost comparison on the 30-dimensional Wave benchmark. We report the GFLOPs required per training update for each method on a logarithmic scale. TensorSpectra is evaluated separately for the base and residual-correction stages. Lower values indicate lower computational cost.}
    \label{fig:computational_cost}
\end{figure}

Furthermore, Table~\ref{tab:wave30_parameters} reports the corresponding model parameter counts. TensorSpectra contains 24,160 parameters in its base model and 12,080 parameters in its residual correction. Despite having more parameters than SINN, TensorSpectra remains substantially more parameter-efficient than PINN, SPINN, and PirateNet. Moreover, its staged optimization strategy requires updating only the newly introduced parameters during residual correction. Together, these results demonstrate that TensorSpectra achieves a favorable balance between model compactness and per-update computational efficiency.

\begin{table}[htbp]
\centering
\caption{Comparison of model parameter counts on the
30-dimensional Wave benchmark. For TensorSpectra,
we display the parameter counts both in base stage and correction stage.}
\label{tab:wave30_parameters}
\small
\begin{tabular}{llr}
\toprule
Method & Configuration & Parameters \\
\midrule
PINN & Width 256, 3 layers & 147,713 \\
SPINN & Width 128, 3 layers, rank 128 & 1,547,392 \\
PirateNet & Width 256, 3 blocks & 731,780 \\
SINN & 256 features & 8,706 \\
\midrule
TensorSpectra (base) & Rank 8, time degree 8 & 24,160 \\
TensorSpectra (correction) & Rank 4, time degree 8 & 12,080 \\
\bottomrule
\end{tabular}
\end{table}

\section{Experimental Configurations}
\label{sec:ex_config}
This section summarizes the main experimental configurations used throughout our numerical studies. Table~\ref{tab:hyperparameters} reports the spectral basis, temporal polynomial degree, tensor rank, training epochs, and learning rates adopted for TensorSpectra across different PDE benchmarks.

\begin{table}[htbp]
\centering
\caption{Main hyperparameter settings of TensorSpectra for different PDE benchmarks.}
\label{tab:hyperparameters}
\small
\begin{tabular}{lccccc}
\toprule
Benchmark & Basis & Time degree & Rank & Epochs & Learning rate \\
\midrule
Poisson
& Fourier
& --
& $4+2$
& $30\mathrm{k}+20\mathrm{k}$
& 1.0e-03 \\

Heat
& Fourier
& 4
& $8+4$
& $30\mathrm{k}+20\mathrm{k}$
& 1.0e-03 \\

Wave
& Fourier
& 8
& $8+4$
& $50\mathrm{k}+50\mathrm{k}$
& 1.0e-03 \\

Allen--Cahn
& Legendre
& --
& $8+4$
& $25\mathrm{k}+25\mathrm{k}$
& 1.0e-03 \\

Reaction
& Legendre
& 8
& $8+4$
& $25\mathrm{k}+25\mathrm{k}$
& 1.0e-03 \\
\bottomrule
\end{tabular}
\end{table}


\section{Visualization of Learned Spectral Factors and CP Amplitudes}
\label{sec:cp-visualization}
We inspect the stored CP representation through its one-dimensional spectral factors and component amplitudes. These views describe how the parameterization distributes coefficients across Fourier modes and CP terms without assembling the full coefficient tensor.

\paragraph{Spectral factor profiles.}
Figure~\ref{fig:cp-spectral-factors} shows factor magnitudes for the real and imaginary CP representations along coordinates $j=1,5,10$. The real CP factors have pronounced peaks near the lowest Fourier modes and smaller magnitudes away from the center. The imaginary CP factors are comparatively flat away from a narrow dip near the zero mode. Within each representation, the curves largely overlap across components and have similar shapes in the three displayed coordinates. These are properties of the individual factors; the multidimensional coefficients also depend on their products and component amplitudes.

\paragraph{Distribution of stored component amplitudes.}
Figure~\ref{fig:cp-rank-importance} displays the normalized amplitudes of the eight stored components. The two largest, Im~4 and Im~3, account for $39.9\%$ and $38.8\%$, respectively, or $78.7\%$ together. Each real CP component accounts for approximately $1.2$--$1.3\%$. The displayed amplitudes are therefore strongly uneven. Their interpretation depends on the factor normalization convention: CP rescaling can transfer magnitude between a component amplitude and its factors without changing the represented tensor. Moreover, overlapping components need not contribute independently to the solution norm. Thus the plot describes amplitude concentration in the stored parameterization; establishing component importance or an effective rank would additionally require reconstruction-error or component-removal measurements.

\begin{figure}[!htbp]
    \centering
    \includegraphics[width=\textwidth]{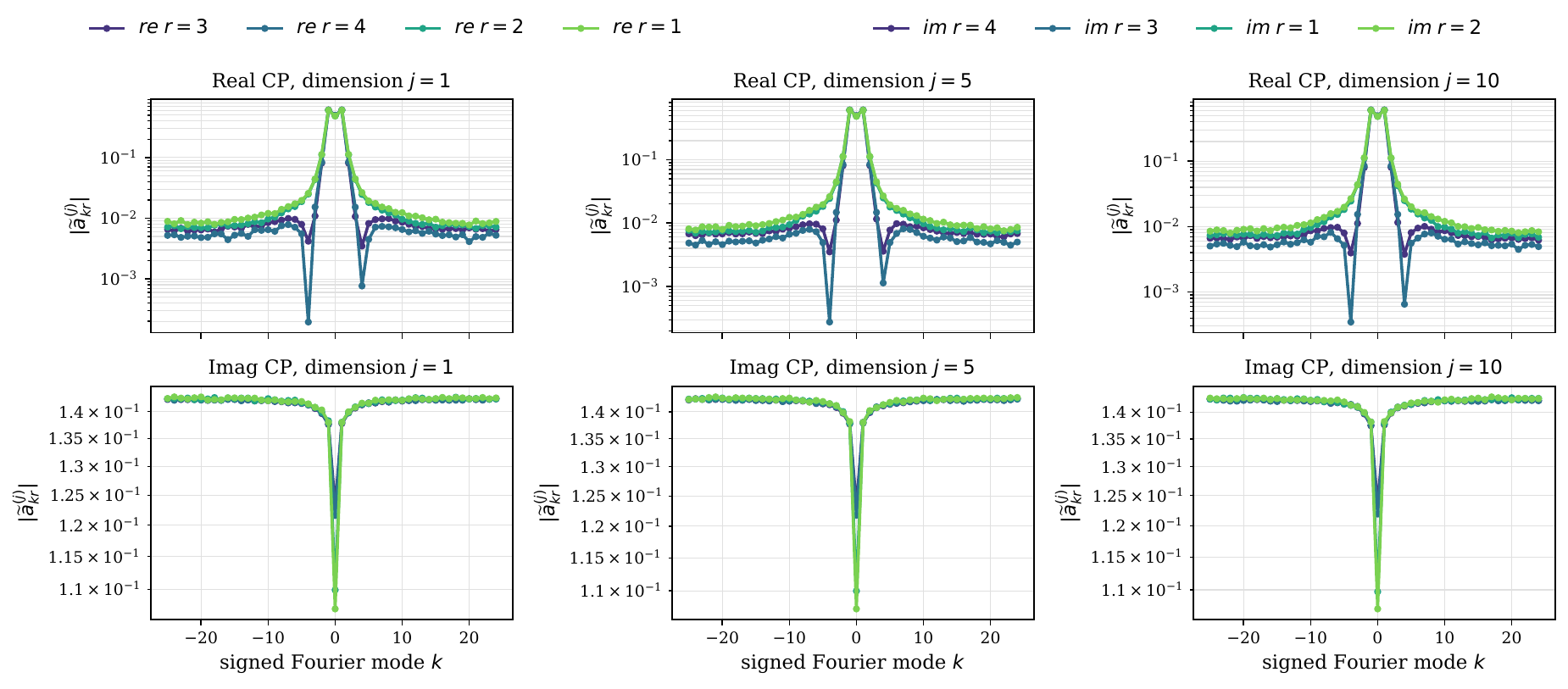}
    \caption{Magnitudes of the stored spectral factors over signed Fourier modes for coordinates $j=1,5,10$. The top and bottom rows show the real and imaginary CP representations, respectively; each curve is a stored component. The two rows use different vertical ranges.}
    \label{fig:cp-spectral-factors}
\end{figure}

\begin{figure}[!htbp]
    \centering
    \includegraphics[width=0.65\linewidth]{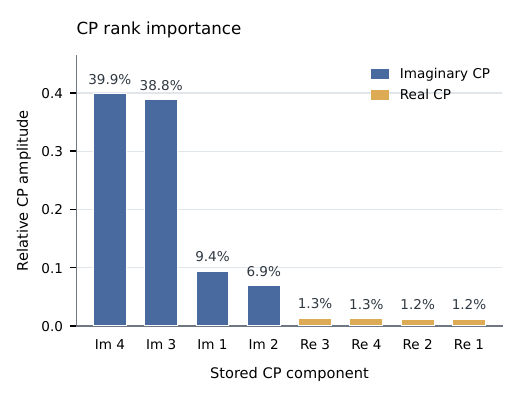}
    \caption{Normalized amplitudes of the eight stored CP components, ordered by magnitude. Blue and orange bars denote imaginary and real CP components, respectively; labels retain the stored component indices. Percentages quantify relative amplitudes in this parameterization, rather than independent contributions to solution accuracy.}
    \label{fig:cp-rank-importance}
\end{figure}

\end{document}